\documentclass[8pt]{article}
\usepackage{graphicx,psfrag,amssymb,amsmath,amsthm}
\usepackage{latexsym, comment}
 \usepackage{color}
\usepackage[margin=0in]{geometry}
\usepackage{makeidx}
 \usepackage{eucal}\makeindex
\usepackage{pdfsync}
\usepackage{hyperref}

\def\beq{\begin{equation}}
\def\eeq{\end{equation}}

\def\dist{{\rm dist}}

\def\bB{\mathbf{B}}

\def\bI{\mathbf{I}}

\def\bR{\mathbf{R}}

\def\bf{\textbf{}}

\def\bl{\mathbf{l}}

\def\bx{\mathbf{x}}

\def\b1{{\boldsymbol{1}}}

\def\cA{\mathcal{A}}
\def\cB{\mathcal{B}}

\def\cD{\mathcal{D}}
\def\cE{\mathcal{E}}
\def\cF{\mathcal{F}}
\def\cG{\mathcal{G}}
\def\cH{\mathcal{H}}

\def\eps{\varepsilon}

\def\cM{\mathcal{M}}
\def\cN{\mathcal{N}}
\def\cO{\mathcal{O}}

\def\cR{\mathcal{R}}
\def\cS{\mathcal{S}}

\def\cW{\mathcal{W}}

\def\cZ{\mathcal{Z}}

\def\tf{{\tilde{f}}}

\def\tZ{\tilde{Z}}

\def\tg{\tilde{g}}

\def\tZ{\tilde{Z}}
\def\tX{\tilde{X}}

\def\eps{\varepsilon}

\def\dist{\text{\rm dist}}

\def\1{\mathbf{1}}
\def\RR{\mathbb R}
\def\NN{\mathbb N}

\def\EE{\mathbb E}
\def\PP{\mathbb P}
\def\cD{\mathcal D}
\def\M{\mathcal M}
\def\N{\mathcal N}
\def\E{\mathcal E}
\def\B{\mathcal B}

\def\R{\mathcal{R}}

\def\cH{\mathcal H}
\def\F{\mathcal F}
\def\cF{\mathcal F}
\def\cG{\mathcal G}
\def\S{\mathcal S}

\def\O{\mathcal O}

\def\cov{\text{Cov}}

\def\an1{A_{n,1}}

\def\an1c{A_{n,1}^c}

\def\hxn1{\hat X_{n,1}}

\def\hzn1{\hat Z_{n,1}}

\def\hsn1{\hat S_{n,1}}

\def\x1{X_1}

\def\var{\text{Var}}

\def\tf{\tilde f}

\def\cov{\text{Cov}}

\def \E{\mathbf{E}}

\def\({\left(}
\def\rt){\right)}

\def \N{\mathbb{N}}

\def \E{\mathbf{E}}

\def \eps{\epsilon}

\def \({\left(}
\def \){\right)}

\def \beq{\begin{equation}}
\def \ee{\end{equation}}
\def \bea{\begin{eqnarray}}
\def \eea{\end{eqnarray}}
\def \bes{\begin{eqnarray*}}
\def \ees{\end{eqnarray*}}

\def \nn{\nonumber}

\newtheorem{theorem}{Theorem}

\newtheorem{lemma}[theorem]{Lemma}

\newtheorem{proposition}[theorem]{Proposition}

\newtheoremstyle{myremark}{}{}{\small\normalfont}{0pt}{\small\scshape}{.}{.5em}{}

\theoremstyle{myremark}

\begin{document}

\title{Supperdiffusions for certain nonuniformly hyperbolic systems}

\author{Luke Mohr
 \and
Hong-Kun Zhang\thanks{Department of Mathematics and Statistics, University of Massachusetts,  Amherst MA 01003, USA. Email: lukemohr@gmail.com, hongkun@math.umass.edu. }}

\date{\today}
\maketitle

\begin{abstract}
We investigate  superdiffusion for stochastic processes generated by  nonuniformly hyperbolic system models, in terms of the convergence of rescaled distributions
to the  normal distribution following the abnormal central limit theorem, which differs from the usual
requirement that the mean square displacement grow asymptotically linearly in time.  We construct a martingale approximation that follows the idea of Doob's decomposition theorem.  We obtain an explicity formula for the superdiffusion constant in terms of the fine structure  that  originates in the phase transitions as well as the geometry of the  configuration  domains of the systems.    Models that satisfy our main assumptions include   chaotic Lorentz gas, Bunimovich stadia, billiards with cusps, and can be apply to other nonuniformly hyperbolic systems with slow correlation decay rates of order $\cO(1/n)$.  \vskip1cm

 {\bf Keywords:} nonstandard central limit theorem, nonuniformly hyperbolic systems,  first return time.
 \centerline{AMS classification numbers: 37D50, 37A25}
\end{abstract}
\tableofcontents
\printindex
\section{Introduction and main results}
\subsection{Introduction}
One of the most fundamental mechanisms in
 nonequilibrium physical systems  is the diffusion process.  Although random processes have been used to model diffusion based on
Einstein's seminal work on Brownian motion \cite{Ein},  it has been realized recently that  many simple deterministic dynamical systems  resemble diffusion to some extent. The theory of dynamical systems has its origin in classical and statistical mechanics through the works of Poincar\'e and
Boltzmann. One of the key aims of  statistical mechanics is to relate the microscopic properties of a fluid  to the transport coefficient which leads to diffusion on a macroscopic
scale.  These include diffusion coefficients, viscosity, and heat conduction.
Being the simplest physical systems resembling diffusion \cite{BSp},   deterministic billiards have attracted much attention since Sinai's seminal work  \cite{Sin}. Limiting laws in classical hyperbolic systems are better understood and proved or almost proved in quite a few cases. However, only recently have these laws
become a main focus of study
for nonuniformly hyperbolic systems, so the development of new techniques to prove limiting laws is of great mathematical
interest.

There
are many physically-motivated variations on billiards, such as Lorentz gas, Bunimovich Stadia, etc.; see \cite{Sin, Bu74,Bu79,BSC90,BSC91, CM} for detailed descriptions.   These systems were proved to  be hyperbolic, ergodic, and mixing.  Many
 mixing  systems have slow (polynomial)
mixing rates which cause weak statistical properties; this situation commonly arises in nonuniformly hyperbolic systems.
The
central limit theorem may fail and affect the convergence to a Brownian motion in a
proper space-time limit (weak-invariance principle) and many other useful approximations by
stochastic processes that play crucial roles in statistical
mechanics. Such systems  exemplify a delicate
transition from regular behavior to chaos. For this reason, they
have attracted considerable interest in mathematical physics
during the past 20 years; see \cite{ABV, BCD, BSC91, C99,C06, CD09,CM07, CZ05a, CZ09, D04,H82,Ma04, P92,Y98,Y99}  and the references therein.

It is somewhat challenging to study hyperbolic systems with singularities, including chaotic billiards. One main reason is that these systems may have singularities that  lead to an unpleasant fragmentation of the phase space. More precisely, any unstable manifold may expand locally, but the singularities may cut its images into many pieces. Some of the resulting pieces are a much
smaller size than the original ones, and this requires a long time to
recover. Moreover, the differential of the map can also be unbounded and/or have unbounded distortion,  which aggravates the analysis.

Let $(T, \cM,    \mu_{\cM})$ be an ergodic transformation of a probability space $(\cM,\mu_{\cM},\cF)$,
with $\cM$ a $d$-dimensional Riemannian manifold and $\cF$ the Borel $\sigma$-algebra on $\cM$. We assume   the map  $T: \cM \rightarrow \cM$ preserves a mixing probability measure
$\mu_{\cM}$ on the
$\sigma$-algebra $\cF$. One can study the statistical properties of a real-valued observable function
$f$ on $\cM$ by defining the sequence of random variables $X_i = f \circ T^i$, $i\geq 1$; this sequence is dependent but
identically distributed due to the invariance of $\mu_{\cM}$. One intuitive line of inquiry is to study whether classical  probability limiting theorems, such as the
central limit theorem, are satisfied for this process. It is often the case that as long as the system in question
exhibits sufficiently chaotic behavior one may expect such limiting theorems to hold.

Much attention is now shifting to open questions concerning  nonuniformly hyperbolic systems with  decay rates of correlations as slow as of order $\cO(n^{-1})$. Indeed, the central limit theorem has been proved for
a variety of these systems, see for example the works of B\'alint, Chernov, Dolgopyat, Gou\"ezel, Liverani, Markarian,  Sz\'asz, and
Varj\'u \cite{BCD, GB, CM,  L96, SV07}.  The techniques for two-dimensional hyperbolic systems that have been previously utilized are often quite
system specific and require geometric calculations based on the particular dynamical system being considered. In this
paper, we utilize a martingale difference decomposition
technique which can be used to prove the Central Limit Theorem (CLT) for  a wide variety of
nonuniformly hyperbolic systems.  In addition to the new martingale approximation method for hyperbolic systems, we also go beyond previous research, including \cite{BCD, GB, SV07}, by (i)  proving the CLT for processes $\{f\circ T^n\}$ generated by more general  observables, that are only required to be H\"{o}older continuous on stable manifolds;  (ii)  we are able to explicitly compute the supperdiffusion constants for Sinai billiards on torus with finite number of free flight channels; (iii) the method developed in this paper is not restricted to billiards, but can also be applied for other non uniformly hyperbolic systems.

Our main goal for this paper is to develop central limit theorems for  certain nonuniformly hyperbolic
systems with slow correlation decay rates of order $\cO(n^{-1})$. The stochastic processes generated by these systems exhibit the super-diffusion phenomenon. The main tools we use in our proofs are martingale
approximation and the martingale central limit theorem, which are presented in depth by Hall and Heyde in
\cite{HH80} and  Helland in \cite{H82}. One advantage of our methods is that we will, in many cases, be able to give explicit expressions
for the super-diffusion constants.
Additionally, we find that our proposed methods are applicable
to a wide variety of hyperbolic systems. Our goal is that this will lead to a more unified approach to studying the statistical
properties of nonuniformly hyperbolic systems with slow decay of correlations.

One major advantage of studying martingales is that, while they are not generally independent sequences of random
variables, their dependence is ``weak enough" that it is possible to generalize results for the i.i.d. case to martingales
with certain additional properties. Furthermore, the following result due to Doob \cite{Doo90} gives the study of
the statistical properties of martingales even greater significance.\\

\noindent\textbf{Doob's decomposition theorem.}
\textit{Let $(\Omega, \F, \PP)$ be a probability space, $\{\F_n\}_{n\in\NN}$ a filtration of $\F$, and $\{X_n\}_{n\in\NN}$ an
adapted stochastic process with $\EE|X_n| < \infty$ for all $n\in\NN$. Then there exists a martingale $\{M_n\}_{n\in\NN}$
and an integrable predictable process $\{A_n\}_{n\in\NN}$ starting with $A_1 = 0$ such that $X_n = M_n + A_n$ for
every $n\in\NN$. This decomposition is unique almost surely.}\\

Note that a process $\{A_n\}_{n\in\NN}$ is predictable if $A_n$ is $\F_{n-1}$-measurable for every $n\ge 2$. The above
result illustrates the usefulness of studying the statistical properties of martingales. If a stochastic process is
adapted to a filtration and each random variable in that process is integrable, then showing the central limit theorem
for that stochastic process may reduce to proving an associated martingale central limit theorem, as long as the
process $\{A_n\}_{n\in\NN}$ is, in a sense, negligibly small.

\subsection{Abnormal CLT for certain stationary Processes}

Let $(\Omega, \cF, \mathbb{P})$ be a probability space, with  $\cF$ being a $\sigma$-algebra.  Let  $\{\cF_n, n\geq 0\}$ be an increasing filtration with
$$\cF_0\subset\cdots\subset \cF_{n-1}\subset\cF_n\subset \cF.$$ We assume $\{X_n, n\geq 0\}$ is a sequence of stationary random variables that is adapted to the filtration $\{\cF_n, n\geq 0\}$.

We first make some specific assumptions on the stochastic  process $\{X_n\}$.
\medskip

\begin{itemize}
\item[(\textbf{A1})] Assume $X_0$ has infinite variance, and $\EE(X_0)=0$.
  Let  $H(t):=\var (X_0\bI_{ |X_0|<t})$, we assume $H(t)$ is a slowly varying function at infinity, i.e,  for all $c>0$, $\lim_{t\to\infty} H(ct)/H(t)=1$. Let  $c_n\in (0,+\infty]$, with $\lim c_n\to\infty$, such that
\beq\label{choosecn}\lim_{n\to\infty} n \mathbb{P}(|X_1|\geq c_n)=0,\,\,\,\,\,\,\,\lim_{n\to\infty}\frac{c_n}{\sqrt{n H(c_n)}}=0.\eeq
We define $X_{n,k}=X_k\cdot\bI_{|X_k|<c_n}$ for any $n\geq 1$ and $k=0,\cdots, n$.
\item[(\textbf{A2})]  There exists $\theta\in (-1,1)$ 
such that for any $n\geq 0$, $k=1,\cdots, n$,
\beq\label{EEXXA12}\EE(X_{n,k}| \cF_{k-1}) = \theta X_{n,k-1}+\cE_{n,k-1},\eeq
where $\cE_{n,k}\in \cF_{k-1}$ is a stationary process such that $\mathbb{E}( \cE_{n,0}^2)<\infty$, $\mathbb{E}(|X_{n,0}\cE_{n,0}|)<\infty$ and
  $$ \sum_{k=0}^{n-1}(n-k)\cov(\cE_{n,0}\cdot\cE_{n,k})=\cO(n).$$\end{itemize}

\begin{theorem}\label{MCLXn0} Let $\{X_n, n\geq 0\}$ be a stationary process that satisfies assumptions \textbf{(A1)-(A2)}. Then the following sequence converges:
\beq\label{ClTZX}\frac{X_1+\cdots+X_n}{\sqrt{ n H( c_n)(1+\theta)/(1-\theta)}}\xrightarrow{d} \N(0,1)\eeq
(in distribution).
Here,  $\N(0,1)$ is a standard normal distribution, and $\theta,  c_n$ are chosen  according to \textbf{(A1)-(A2)}.
\end{theorem}

Next we prove the Invariance Principle.
Let $J=[0,1]$. Let $D(J)$ be the space of right continuous (with left limits ) real valued functions on $J$, endowed with the Skorohod $J_1$ topology:
that is, for any small $\rho>0$, we say two functions $u$ and $v$ in $D(J)$ are $\rho$-close if there exists
$\lambda:[0,1] \to [0,1]$ such that
\begin{itemize}
 \item[(1)]$\lambda(0)=0$, $\lambda(1)=1$, and $\lambda$ is increasing;
     \item[(2)] $\forall t\in [0,1]$, $|\lambda(t)-t|\leq \rho$;
\item[(3)] $\forall t\in[0,1], |u(\lambda(t))-v(\lambda(t))|\leq \rho$.
\end{itemize}
We  define a random function such that for any $t\in [0,1]$, $n\geq 1$,
\beq\label{Wt12}
W_n(X, t)=\sum_{k=1}^{[tn]}\frac{X_k}{\sqrt{n H(c_n)}},
\eeq
and we denote $W$ as the standard Brownian motion on $D(J)$.

\begin{theorem}\label{IP}
Suppose that $X_{n}$ satisfies assumptions \textbf{(A1)-(A2)}.
Then $W_n(X,\cdot)\to \sqrt{\frac{1+\theta}{1-\theta}} W(\cdot)$ weakly on $D(J)$ as $n\to\infty$.
\end{theorem}

Proofs of Theorem 1 and Theorem 2 can be found in Section \ref{profthm12}.

\subsection{Abnormal CLT for Dynamical Systems}

We consider a dynamical system $(T, \cM, \mu_{\cM})$ and assume that the map $T$ is  hyperbolic with singularity $\cS\subset \cM$, as defined by Katok and Strelcyn \cite{KS}, and $\mu_{\cM}$ is a mixing SRB measure. To investigate the statistical properties of $(T,\cM,\mu_{\cM})$, we introduce an induced system $(F, M)$. Let $\cD=\cup_{i\in I}\Omega_{i}$ (card$\,I<\infty$) be a finite union
of some connected components of $\cM\setminus \cS$. For any $x\in \cD$, let
\beq \label{defR}
   \mathcal{R}(x)=\min\{n\geq 1\colon T^n(x) \in \cD,\ \
   T^m(x)\notin \cS,\ m=1, \ldots, n-1\}
\eeq
be the first return time function. We denote by $\cN_1\subset\cD$ the set of
points that never return to $\cD$.  For each $n\geq 1$, the ``level'' set $\cD_n = \{\mathcal{R}(x)=n\} \subset \cD$
is open, and if $\cD_n \neq \emptyset$ then $T^n$ is a diffeomorphism
of $\cD_n$ onto $T^n(\cD_n) \subset \cD$. We denote by $F$ the first
return map, i.e.,
$$
   F(x) = T^n(x)\qquad \forall x\in \cD_n,\quad n\geq 1.
$$
It is easy to see that $F$ is a diffeomorphism of the open set
$\cD^{+}=\cup_{n\geq 1} \cD_n$ onto the open set $\cD^-=\cup_{n\geq 1}
T^n(\cD_n)$. The inverse map $F^{-1}$ is defined on $\cD^- \subset
\cD$ and maps it back to $\cD^+$. Let $M=\overline{\cD}$ denote the
closure of $\cD$, and for each $n\geq 1$ let
\beq\label{def:Mn} M_n = \overline{\cD}_n.\eeq

We set
\[
  S_1 = M\setminus\cD^+ = \cN_1\cap\partial \cD
\]
and
\[
  S_{-1} = M\setminus\cD^- = \cN_{-1}\cap\partial \cD,
\]
where $\cN_{-1}\subset \cD$ denotes the set of points that never return
to $\cD$ under the iterations of $T^{-1}$. We assume that both $S_1$
and $S_{-1}$ are finite or countable unions of smooth compact curves.
The sets $S_{\pm 1}$ play the role of singularities for the induced
maps $F^{\pm 1}$. We assume that the map $F$ restricted to any level
set $\cD_n$ can be extended by continuity to its boundary $\partial
\cD_n$, but the extensions to $\partial \cD_n \cap \partial \cD_m$ for
$n\neq m$ need not agree.

We assume that $\mu_{\cM}(\cD)>0$.  By the ergodicity of $\mu_{\cM}$ we have
\beq \label{cMMm}
  \cM = \bigcup_{n\geq 1}\bigcup_{m=0}^{n-1} T^m M_n\qquad {\rm (mod\ 0)}.
\eeq
The first return
map $F$ preserves the measure $\mu_{\cM}$ conditioned on $M$; we denote it by $\mu$.  We note that $\int_M \mathcal{R}\, d\mu=\mu_{\cM}(M)^{-1}$ by the Kac theorem.
The measure $\mu$ is ergodic, and we assume that it is also mixing.

We assume the support of $\mu$ has a measurable foliation $\cW^u$ of unstable manifolds of the map $F$, such that for any $W^u\in \cW^u$ and any $x,y\in W^u$, we have
\beq\label{A4} d(F^{-n}(x),F^{-n}(y))<C\eta^n,\eeq
for some constant $C>0$ and $\eta\in (0,1)$.

 Assume $\cF$ is the Borel $\sigma$-algebra on the region $M\subset \mathbb{R}^2$. Let $\cA_0=\sigma(\cR)$ be the smallest $\sigma$-algebra generated by the return time function $\cR$, or equivalently, generated by a (finite or) countable partition $\{\Omega_n, n\geq 1\}$ of $M\setminus S_1$. We define the $\sigma$-algebra $\cF_{0}=\sigma(\cR\circ F^{k}, k\leq 0)$, and let $\cF_{-1}$ be the trivial $\sigma$-algebra.   We define, for $n\geq 1$
 \beq\label{defFn}
 \cF_n=\sigma(\cR\circ F^k, k\in (-\infty, n]).
 \eeq
One can check that  $\{\cF_n, n\geq 0\}$ is an increasing filtration with
$$\cF_{n-1}\subset\cF_n\subset \cF.$$
To get a clear picture of $\cF_0$, note that any unstable manifold $W^u\in \cF_0$. Indeed $\cF_0$ is the smallest $\sigma$-algebra generated by the foliation $\cW^u$,   any observable that is constant on each unstable manifold must belong to $\cF_0$. In particular, one can check that for the return time function $\mathcal{R}\circ F^n\in \cF_n$, for $n\geq 0$.

For any observable $f:\cM\to \mathbb{R}$ with finite expectation, we define $\EE_n(f)=\EE(f|\cF_n)$ and  its induced function $\tf: M\to\mathbb{R}$ such that $\tf(x)=f(x)+\cdots+f(T^{n-1}x)$ for any $x\in M_n$.
 For any observable function $f$ on $\cM$, we  denote by $\S_n f$ and $S_n  \tf$ the Birkhoff sums of $f$ and the induced function $\tf$, respectively:
\beq\label{inducesum}
\cS_n  f =  f+ f\circ T+\cdots + f\circ T^{n-1},\,\,\,\,\,\,\text { and }\,\,\,\,\,\,S_n  \tf =  \tf+ \tf\circ F+\cdots + \tf\circ F^{n-1}.\eeq

 To prove the central limit theorem for the process $\{\tf\circ F^n\}$, we need to introduce the class of observables that we will use to study statistical properties.
For any $\gamma\in (0,1]$, let $\cH_{\gamma}$	be the set of all  functions $g\in L_{\infty}(\cM,\mu_{\cM})$  such that  for any $x,y\in \cM$, \beq \label{DHC--} 	|g(x) - g(y)| \leq \|g\|_{\gamma} d(x,y)^{\gamma},\eeq
where
$$\|g\|_{\gamma}:= \sup_{ x, y\in \cM}\frac{|g(x)-g(y)|}{d(x,y)^{\gamma}}<\infty.$$
If $\|g\|_{\infty}<\infty$, we also  define \beq\label{defCgamma}\|g\|_{C^{\gamma}}:=\|g\|_{\infty}+\|g\|_{\gamma}.\eeq

\begin{theorem}\label{MCLXn2}  Let $f\in\cH_{\gamma}$ with exponent $\gamma\in (0,1)$. We assume the process $X_n=\mathbb{E}_n(\tf\circ F^n)-\EE(\tf)$  satisfies assumptions \textbf{(A1)-(A2)}, and the process $\xi_{n,k}:=(\tf\cdot\bI_{|\tf|\leq c_n}-\E_0(\tf\cdot\bI_{|\tf|\leq c_n}))\circ F^k$ satisfies\\

\noindent(\textbf{A3})\, \, $ \sum_{k=0}^{m-1}(m-k)\text{\bf{Cov}}(\xi_{n,0}\cdot\xi_{n,k})=\cO(m),$ for any $m\leq n.$\\

Then the following sequence converges in distribution as $n\to\infty$:
\beq\label{ClTZf}\frac{S_n \tf-n\mu(\tf)}{\sqrt{\tfrac{1+\theta}{1-\theta}\cdot  n H( c_n)}}\xrightarrow{d} \N(0,1).\eeq
Here $\theta,  c_n$ are chosen as in \textbf{(A1)-(A2)}. \end{theorem}

Next we prove the Invariance Principle and  form a random function such that
\beq\label{Wt4}
W_n(\tf,F,\mu, t)=\sum_{k=1}^{[tn]}\frac{\tf\circ F^k-\mu(\tf)}{\sqrt{n H(c_n)}},
\eeq
where we denote $W$ as the standard Brownian motion on $D(J)$.
\begin{theorem}\label{IPY}  Let $f\in\cH_{\gamma}$ with exponent $\gamma\in (0,1)$. We assume the process $X_n=\mathbb{E}_n(\tf\circ F^n)-\EE(\tf)$  and $\{\xi_{n,k}\}$ satisfy assumptions \textbf{(A1)-(A3)}. Then
\beq\label{WnIPt} W_n(\tf,F,\mu, \cdot)\to \sqrt{\tfrac{1+\theta}{1-\theta}}\cdot W(\cdot)\eeq
 weakly on $D(J)$ as $n\to\infty$.\end{theorem}

More generally, we have the following results for the original system $(\cM,T,\mu_{\cM})$.
For any observable $f:\cM\to \mathbb{R}$ with finite expectation, we denote the induced system by $(M,F,\mu)$ and the induced function by $\tf: M\to\mathbb{R}$.

\begin{theorem}\label{IPYf} Let $f$ be an observable on $\cM$ where $f\in\cH_{\gamma}$ with exponent $\gamma\in (0,1)$. We assume the process $X_n=\mathbb{E}_n(\tf\circ F^n)-\EE(\tf)$ and $\{\xi_{n,k}\}$ satisfy assumptions \textbf{(A1)-(A3)}. Then $$W_n(f,T,\mu_{\cM},\cdot)\to \sqrt{ \mu_{\cM}(M) \frac{1+\theta}{1-\theta}} \cdot W(\cdot)$$ weakly on $D(J)$ as $n\to\infty$.
\end{theorem}

As a consequence of the above invariance principle, we have the following central limit theorem.

\begin{theorem}\label{MCLXn1}  For any observable $f\in H_{\gamma}$ with $\gamma\in (0,1)$, we assume  assumptions \textbf{(A1)-(A3)}. Then the following sequence converges in distribution:
\beq\label{ClTZ}\frac{f+\cdots +f\circ T^{n-1}-n\mu_{\cM}(f)}{\sqrt{ \frac{1+\theta}{1-\theta}\cdot n H( c_n)\mu_{\cM}(M)}}\xrightarrow{d} \N(0,1).\eeq
\end{theorem}

\vspace{0.5cm}

The convergence in (\ref{ClTZ}) means precisely that for any $z\in \mathbb{R}$,
$$\mu_{\cM}\left(\frac{\cS_n f-n\mu_{\cM}(f)}{\sqrt{ \frac{1+\theta}{1-\theta}\cdot n H(c_n)\mu_{\cM}(M)}}<z\right)\to \int_{-\infty}^z \frac{1}{\sqrt{2\pi}} e^{-\frac{x^2}{2}} dx,$$ as $n\to\infty$. Since the map $(T,\mu_{\cM})$ is mixing, the limit law (\ref{ClTZ}) holds
   if we  replace $\mu_{\cM}$ with any probability measure that is absolutely continuous with respect
to $\mu_{\cM}$; see Section 4.2 in \cite{Eagleson}.

An analogue of Theorem \ref{MCLXn1} was proved for some
billiard models where correlations decay as $\cO(1/n)$ for the Bunimovich
stadium, the
Lorentz gas with infinite horizon, and billiards with cusps; see \cite{BCD,GB,SV07}. The main goal of this paper is to provide a new method using martingale approximation, and performing a  simplified and unified  study on the central limit theorem for systems with similar properties. Moreover, Theorem \ref{MCLXn1} provides the precise formula for the diffusion constants (i.e. coefficient of the term $n\ln n$). In particular, this allows us to obtain the diffusion constant for Lorentz gas with infinite horizon, which is new to our knowledge.

 \begin{proposition}\label{Thm3}   Theorem \ref{IPYf} follows from Theorem \ref{IPY}.
\end{proposition}

The proof of Proposition \ref{Thm3} can be found in Section \ref{Prop7}. Thus, it is enough for us to prove Theorem \ref{IPY} for the induced systems.

In Section \ref{profthm12}, we prove a central limit theorem for mixing stationary processes by using martingale approximations. These are inspired by and have direct
applications to the nonuniformly hyperbolic
systems we subsequently study more in-depth. The processes have weak dependence and special
assumptions on their conditional expectations that arise naturally in these systems.  Theorem \ref{MCLXn2} and Theorem \ref{IPY} are proved  in Section 3.  Section 5 lists some sufficient conditions for estimating diffusion coefficients that will be used in our applications. In Section 6, we study certain classes of nonuniformly hyperbolic billiards and
show that we can apply the main results, especially  Theorem \ref{A5con} to each of these systems. Although we only concentrate on stadia, billiards with cusps, and semi-dispersing billiards, the main theorems have legitimate applications to rather general hyperbolic systems.

Throughout this paper we will use the following conventions:
positive and finite global constants whose value is unimportant will be denoted by $c$, $c_1, c_2$, ... or
$C$, $C_1$, $C_2$, ...., etc.	 These letters
may denote different values in different equations
throughout the paper.  Let $d(\,,\,)$ be the distance
function on $ \cM \times \cM $ induced by the Riemannian
metric  in $ \cM $. For any smooth curve  $W$ in $ \cM $, we
denote by $|W|$ its Lebesgue length.  For any measurable set $A\subset \cM$, we denote
$\bI_{A}$ as the indicator function of the set $A$. Given two sequences $A_n$ and $B_n$, we use the notation $A_n\asymp B_n$ to indicate that there exists two uniform constants $C_1<C_2$ such that $C_1B_n<A_n<C_2B_n$.\\

\section{Proof of Theorem 1 and Theorem 2}\label{profthm12}

\subsection{Martingale approximations}

In this section, we consider the stationary process $\{X_n\}$ adapted to the filtration $\{\cF_n\}$, with $\EE(X_n)=0$.
We say that a random variable $X$ {\it{belongs to the domain of attraction of the normal distribution}} if there exists $b_n>0$ such that $$\frac{X_1+\cdots +X_n }{b_n}\to \N(0,1)$$
in distribution as $n\to\infty$.

Note that assumption \textbf{(A1)} provides a sufficient condition for the stationary process $\{X_n\}$ to be in the domain of attraction of the normal distribution with $a_n=0$, so our first goal is to find $b_n$.

We  let $\bI_A$ denote the indicator of an
event $A$. To overcome the difficulty of the infinite variance property, we next define a triangular array $ \{X_{n,k}\,:\, n\geq 1, k=0, \ldots,n \}$ adapted to $\{\cF_{n,k}, n\geq 1, k=0, \ldots, n\}$, such that
$$X_{n,k}= (X_k \cdot \bI_{|X_k|< c_n}),\,\,\,\,\cF_{n,k}=\cF_k.$$
Clearly, for each $n\geq 0$, the row of random variables $\{X_{n,k}, k=0, \ldots, n\}$ are identically distributed with
\beq\label{EEVar} \var(X_{n,k})=H(c_n).\eeq

The sequence $\{X_{n,k}, k=0, \ldots, n\}$ is said to admit a co-boundary if there is a stationary sequence of martingale differences $Z_k$ and another stationary process $d_k$ for which $X_{n,k}=Z_k+d_k-d_{k-1}$ for all $k=1,\ldots, n$, in which case $S_n=M_n+Y_n$, with $S_n=\sum_{k=1}^n X_{n,k}$, $M_n=\sum_{k=1}^n Z_k$, and $Y_n=d_n-d_0$. Here, $M_n$ is a martingale.
Next, we construct a simple martingale difference approximation for the process $\{X_n\}$.

\begin{lemma}\label{martingale} There exists a  strictly stationary, ergodic martingale difference process array $\{Z_{n,k}, n\geq 1, k=1, \ldots, n\}$ with respect to the filtration $\{\cF_n, n\geq 0\}$ such that $X_{n,k}=Z_{n,k}+Y_{n,k}$, with $Y_{n,k}:=\EE(X_{n,k}|\cF_{k-1})$.\end{lemma}
\begin{proof}

We define, for $n\geq 1$ and $k=1,\ldots,n$,
\[Y_{n,k}:=\EE(X_{n,k}|\cF_{k-1}).\]
Then $\{Y_{n,k}, k=1, \ldots, n\}$ is a predictable process adapted to the filtration $\{\cF_{k-1}, k=1, \ldots, n\}$.
We also define a stationary sequence
\beq\label{defZn}Z_{n,k}=X_{n,k}-\EE(X_{n,k}|\cF_{k-1}).\eeq
In addition,  we define  $Z_{0,0}=X_0$.
One can check that this process is stationary. Note that,  in the probability space $(\Omega,\mathbb{P},\cF)$,  $\{Z_{n,k}, k=1, \ldots, n\}$ is  adapted to the filtration $\{\cF_{k}, k\geq 0\}$. Moreover,  we have
\begin{align*}
\EE(Z_{n,k}|\cF_{k-1})&=\EE(X_{n,k}-\EE(X_{n,k}|\cF_{k-1}))|\cF_{k-1})\\
&=\EE(X_{n,k}|\cF_{k-1}) - \EE(\EE(X_{n,k}|\cF_{k-1})|\cF_{k-1})=0.\end{align*}
Thus, $\{Z_{n,k}\}$ is a martingale difference sequence array. The ergodicity of $\{Z_{n,k}\}$ follows from that of $\{X_n\}$.

Moreover, we can check that for any $n\in \mathbb{Z}$, $k=1,\cdots, n$, we have
$$X_{n,k}= Z_{n,k}+Y_{n,k}.$$
\end{proof}

By Lemma \ref{martingale}, $\{Z_{n,k}, n\geq 1, k=1, \ldots, n\}$ is a strictly stationary, ergodic martingale difference array with respect to the filtration $\{\cF_k, k\geq 0\}$. Moreover, by the stationary property and assumption (\textbf{A2}), we have for any $k\geq 1$ that
\beq\label{Xnktheta}\EE(X_{n,k+1}|\cF_k)=\theta X_{n,k}+\cE_{n,k}.\eeq
 We define the partial sums \[S_n=\sum_{i=1}^n X_{n,i}\] and \[M_n=\sum_{i=1}^n Z_{n,i}.\] We then have
\beq\label{XnMn12}
(1-\theta)S_n=M_n+\theta(X_{n,0}-X_{n,n})+\sum_{k=0}^{n-1}\cE_{n,k}.
\eeq

 Our next goal is to investigate the CLT for the martingale difference $\{Z_{n,k}\}$.
  The following lemma was proved by McLeish in \cite{Mc7}.

\begin{lemma}\label{Helland} Let $\{\tZ_{n,k}\}$ be a martingale difference array adapted to the filtration array $\{\cF_{n,k}\}$. For $n\to\infty$, assume there exists $\sigma>0$ such that:
\begin{enumerate}
\item $\EE(\max_{1\leq k\leq n} |\tZ_{n,k}|^2)\to 0$;
\item $\sum_{k=1}^n \tZ_{n,k}^2\to \sigma^2$, in probability.
\end{enumerate}
Then $\tZ_{n,1}+\cdots+\tZ_{n,n}\to \N(0,\sigma^2)$ in distribution.
\end{lemma}


In order to relax the requirement to a martingale difference sequence with infinite variance we need to prove an abnormal central limit theorem.

\begin{proposition}\label{MCLZn} Let $\{Z_{n,k}\}$ be a martingale difference array defined as in (\ref{defZn}) and adapted to the filtration array $\{\cF_{n,k}\}$.
Then, for $n\to\infty$ and any $k=1, \ldots, n$,
\[
\frac{Z_{n,1}+\cdots +Z_{n,n}}{{\sqrt{ (1-\theta^2) nH( c_n) }}}\to \N(0,1)\]
in distribution.
\end{proposition}

\begin{proof}

To apply the last lemma, we choose $ c_n$ according to (\ref{choosecn}) and
 define an array
\[
\tZ_{n,k}=\frac{Z_{n,k}}{\sqrt{nH( c_n) }}
\]
for $n\geq 1$ amd $k=1, \ldots, n$. We claim that  $\{\tZ_{n,k}, n\geq 1, k=1, \ldots, n\}$ is a martingale difference array that satisfies the assumptions of Lemma \ref{Helland}.
Let $\cF_{n,k}=\cF_k$ for any $n\geq 1$ and $k=1, \ldots, n$. Clearly, $\tZ_{n,k}$ is $\cF_{n,k}$ measurable and, for each fixed $n\geq 1$, $\{\tZ_{n,k}, k=1, \ldots, n\}$ is a martingale difference sequence.
Moreover,
\begin{align*}
|\tZ_{n,k}|^2&=\frac{1}{nH( c_n) }\left(|X_{n,k}-\EE(X_{n,k}|\cF_{k-1})|^2\right).\end{align*}
By assumption, we know that $\EE( |\tilde Z_{n,k}|^2)<\infty$.

We claim that, for any stationary triangular array $\eta_{n,k}$ with finite mean,
\beq\label{claim}\lim_{n\to\infty}\frac{1}{n}\EE\left(\max_{1\leq k\leq n} \eta_{n,k}\right)=0.\eeq
To prove this claim, first choose a large $N>1$. We denote
$$\eta_{n,k}=\eta_{n,k}\cdot\bI_{\eta_{n,k}<N}+\eta_{n,k}\cdot \bI_{\eta_{n,k}\geq N}$$ and $\xi_{n,k}=\eta_{n,k}\cdot \bI_{\eta_{n,k}>N}$.  Note that
\begin{align*}
\EE\left(\max_{1\leq k\leq n}\xi_{n,k}\right)&=\int_{0}^{\infty}\mathbb{P}\left(\max_{1\leq k\leq n}\xi_{n,k}>x\right)\,dx\\
&\leq \sum_{k=1}^n \int_{0}^{\infty}\mathbb{P}(\xi_{n,k}>x)\,dx\\
&=n \int_{0}^{\infty}\mathbb{P}(\xi_{n,1}>x)\,dx=n\EE(\xi_{n,1}).
\end{align*}
Thus, we have
\[
\frac{1}{n}\EE\left(\max_{1\leq k\leq n}\eta_{n,k}\right)\leq \frac{N}{n}+\EE(\xi_{n,1})=\frac{N}{n}+\EE(\eta_{n,1}\cdot\bI_{\eta_{n,1}>N}).
\]
Since $\EE(\eta_{n,1})<\infty$, and we can choose $N$ to be arbitrarily large, the right hand side of the above equation approaches $0$ as $n\to\infty$. This completes the proof of claim (\ref{claim}).

Now, take $\eta_{n,k}=Z_{n,k}^2/H(c_n)$. Then $\EE(\eta_{n,k})<\infty$ and
\[
\lim_{n\to\infty}\frac{1}{n}\EE\left(\max_{1\leq k\leq n}\eta_{n,k}\right)=0.
\]
Therefore,
\[
\EE\left(\max_{1\leq k\leq n}\tilde Z_{n,k}^2\right)=\frac{1}{n}\EE\left(\max_{1\leq k\leq n}\eta_{n,k}\right)\to 0
\]
as $n\to\infty$, which implies that $\EE(\max_{1\leq k\leq n}\tilde Z_{n,k}^2)$ is uniformly bounded and verifies item 1\ of Lemma \ref{Helland}.

Note that (\ref{Xnktheta}) implies that, for any $k\geq 0$,
\[
\EE(X_{n,k+1}|\cF_k)=\theta X_{n,k}+\cE_{n,k}.
\]
We then have
\begin{align*}
nH(c_n)\EE(\tZ_{n,k+1}^2)&=\EE(\EE_k(Z_{n,k+1}^2))\\
&=\EE(\EE_k(X_{n,k+1}^2-2X_{n,k+1}\EE_k(X_{n,k+1})+\EE_k(X_{n,k+1})^2))\\
&=\EE(X_{n,k+1}^2)-\EE(\EE_k(X_{n,k+1}))^2\\
&=\EE(X_{n,k+1}^2)-\EE((\theta X_{n,k}+\cE_{n,k})^2)\\
&=(1-\theta^2)H(c_n)+C_{n,k},\end{align*}
where we use the fact that $H( c_n)=\EE(X_{n,k}^2)=\EE(X_{n,k+1}^2)$, and we have denoted
\[
C_{n,k}=2\theta\EE(\cE_{n,k}\cdot X_{n,k})+\EE(\cE_{n,0}^2)=2\theta\EE(\cE_{n,0}\cdot X_0\cdot \bI_{|X_0|<c_n})+\EE(\cE_{n,0}^2).
\]
It follows from (\textbf{A2}) that
$C_{n,k}$ is uniformly bounded. This leads to the following equality:
\begin{align}\label{Hbarcn}
\sum_{k=1}^n\EE\left(\tZ_{n,k}^2\right)=\EE\left(\sum_{k=1}^{n}\EE_k( \tZ_{n,k}^2)\right)&= (1-\theta^2)+\frac{\sum_{k=0}^{n-1}C_{n,k}}{nH( c_n) },\end{align}
which implies that, as $n\to\infty$,
\[
\sum_{k=1}^n\tZ_{n,k}^2\to 1-\theta^2>0
\]
in mean. Therefore, the sequence also converges in probability. This verifies item 2.\ of Lemma \ref{Helland}.

We can now apply Lemma \ref{Helland} and obtain that, as $n\to\infty$,
\[
\sum_{k=1}^n \tZ_{n,k}\to \N(0,1-\theta^2)
\]
in distribution.
\end{proof}
\subsection{Proof of Theorem \ref{MCLXn0} and Theorem \ref{IP}.}

Let $\{Z_{n,k}, n\geq 0, k=1,\ldots, n\}$ be the martingale difference array defined as in (\ref{defZn}).
We denote by $M_n=Z_{n,1}+
\cdots+Z_{n,n}$ the corresponding martingale adapted to the filtration $\{\cF_n, n\geq 1\}$. Now we can use Proposition \ref{MCLZn}, which implies that,
 as $n\to \infty$,
$$\frac{M_n}{{\sqrt{ (1-\theta^2)nH( c_n)}}}\to \N(0,1)$$
in distribution.

 According to assumptions \textbf{(A1)-(A2)}, we know that the partial sum $S_n$ of $\{X_{n,k}, k=1, \ldots, n\}$ can be written in terms of a martingale $M_n$ and an error term:
\beq\label{XnMn15}
S_n:=X_{n,1}+\cdots+ X_{n,n} =\sum_{i=1}^n \EE(X_{n,i}|\cF_{i-1})+Z_{n,1}+\cdots+ Z_{n,n}=M_n+ \theta\sum_{i=0}^{n-1} X_{n,i}+\sum_{i=0}^{n-1}\cE_{n,i}.
\eeq
This implies that
\beq\label{XnMn2}
(1-\theta)S_n=M_n+ \sum_{i=0}^{n-1}\cE_{n,i}+\theta(X_{n,n}-X_{n,0}).
\eeq

Therefore, as $n\to\infty$,
\begin{align}\label{Sntheta}
\frac{(1-\theta)S_n+\theta (X_{n,n}-X_{n,0})-\sum_{i=0}^{n-1} \cE_{n,i}}{{\sqrt{(1-\theta^2) nH( c_n)}}}=\frac{M_n}{{\sqrt{(1-\theta^2) nH( c_n)}}}\to \N(0,1)
\end{align}
in distribution.

One can easily check that $$ \frac{\theta(X_{n,n}-X_{n,0})}{{\sqrt{ (1-\theta^2) nH( c_n) }}}\to 0$$ in $L_1$ and, thus, also in distribution.
Moreover, we claim that
\[\frac{1}{\sqrt{(1-\theta^2) nH( c_n)}}\sum_{i=0}^{n-1} \cE_{n,i}\to 0
\]
in distribution. We use assumption \textbf{(A2)}, which implies that
\[
\var\left(\sum_{i=0}^{n-1} \cE_{n,i}\right)=n\var(\cE_{n,0})+2\sum_{k=1}^{n-1}(n-k)\cov(\cE_{n,0}\cdot\cE_{n,k})=\cO(n).
\]
Furthermore, $\lim_{n\to\infty} H(c_n)=\infty$ since $X_0$ has infinite variance.
  This implies that
$$\frac{1}{\sqrt{(1-\theta^2) nH( c_n)}}\sum_{i=0}^{n-1} \cE_{n,i}\to 0$$
in $L_2$.

We have shown that
\begin{align*}
\frac{(1-\theta)S_n}{\sqrt{(1-\theta^2) nH( c_n)}}\to \N(0,1)
\end{align*}
in distribution.
This implies that \begin{align*}
\frac{S_n}{\sqrt{ nH( c_n)(1+\theta)/(1-\theta)}}\to \N(0,1)
\end{align*}
in distribution.
In particular, this also implies that
$$\frac{1}{\sqrt{  nH( c_n)(1+\theta)/(1-\theta)}} \sum_{k=1}^n X_k \cdot (\bI_{|X_{k}|< c_n})\to \N(0,1)$$
in distribution.

Note that, by (\ref{choosecn}),
\[
\sum_{k=1}^n \mathbb{P}(X_k\neq X_k\cdot (\bI_{|X_{k}|< c_n}))\leq \sum_{k=1}^n\mathbb{P}( |X_{0}|> c_n)= n \mathbb{P}(|X_{0}|> c_n)\to 0
\]
as $n\to\infty$.
By the Borel-Cantelli Lemma, we have
\[\mu(X_k\neq X_k\cdot \bI_{|X_{k}|< c_n},\,\, \text{i.o.}) = 0.
\]
Combining the above facts, we have show that
\[
\frac{1}{\sqrt{ nH( c_n)(1+\theta)/(1-\theta) }} \sum_{k=1}^n X_k\to \N(0,1)
\]
in distribution, as $n\to\infty$.  This completes the proof of the Theorem  \ref{MCLXn0}.

\vspace{0.7cm}

To prove Theorem \ref{IP}, we use the following Lemma by McLeish in \cite{Mc7}.
\begin{lemma}\label{martingaleIP}
Suppose $\tZ_{n,k}$ is a martingale difference array satisfying\\
(a) $\lim_{n\to\infty}\EE(|\max_{1\leq k\leq [nt]} \tZ_{n,k}|^2)= 0$;\\
(b) $\lim_{n\to\infty}\sum_{n=1}^{[nt]}\tZ^2_{n,k}=t$ for each $t\in J$.\\
Then $W_n(\tZ,\cdot)\to W$ weakly on $D(J)$.
\end{lemma}

Let $\{Z_{{n,k}}, n\geq 0, k=1,\ldots, n\}$ be the martingale difference array defined as in (\ref{defZn}).
We denote by $M_n=Z_{{n,1}}+
\cdots+Z_{{n,n}}$ the corresponding martingale adapted to the filtration $\{\cF_n, n\geq 1\}$.

Note that the two conditions in Lemma \ref{martingaleIP} are very similar to those of Lemma \ref{Helland}. Consequently, the verification of conditions (a) and (b) is almost identical to those in Lemma \ref{Helland}, which we will omit here.
 Thus, for $n\to \infty$,  we have
  $W_n(Z,t)\to W(t)$ weakly on $D(J)$:
\[
W_n(Z,t)=\sum_{k=1}^{[nt]}\frac{Z_{{n,k}}}{\sqrt{n \var{Z_{{n,k}}}} } =\frac{M_{[nt]}}{{\sqrt{ (1-\theta^2)n H( c_{n})}}} \to W(t).
\]
Using (\ref{Sntheta}), as $n\to\infty$ we know that
\beq\label{XnMn1}
 W_n(X,t)=\sigma\sum_{k=1}^{[nt]}\frac{X_{{n,k}}}{n\sigma \var{X_{{n,k}}}} \to\sigma W(t),
\eeq
weakly on $D(J)$ as $n\to\infty$,
where $\sigma=\sqrt{ \frac{1+\theta}{1-\theta}}$. This completes the proof of Theorem \ref{IP}.

\section{Proof of Theorem \ref{MCLXn2} and Theorem \ref{IPY}.}
We consider the filtration $\{\cF_n, n\geq 0\}$ as defined in (\ref{defFn}) and its remark.
For any observable $f:\cM\to \mathbb{R}$,  we define its induced function $\tf: M\to\mathbb{R}$, such that
\[
\tf(x)=f(x)+\cdots+f(T^{n-1}x)\]
 for any $x\in M_n$, with $n\geq 1$. Let  $\EE_k(\tf)=\EE(\tf|\cF_k)$, for $k\geq 0$.

  Let  $H(t):=\var (\tf\cdot \bI_{ |\tf|<t})$, we assume $H(t)$ is a slowly varying function at infinity. Let  $c_n\in (0,+\infty]$, with $\lim c_n\to\infty$, such that
\beq\label{choosecnf}\lim_{n\to\infty} n \mathbb{P}(|\tf|\geq c_n)=0,\,\,\,\,\,\,\,\lim_{n\to\infty}\frac{c_n}{\sqrt{n H(c_n)}}=0.\eeq

Note that, by (\ref{choosecn}),
\[
\sum_{k=1}^n \mathbb{P}(\tf\circ F^k\neq \tf\circ F^k\cdot (\bI_{|\tf\circ F^k|< c_n}))\leq \sum_{k=1}^n\mathbb{P}( |\tf|> c_n)= n \mathbb{P}(|\tf|> c_n)\to 0
\]
as $n\to\infty$.
By the Borel-Cantelli Lemma, we have
\[\mu(\tf\neq \tf\cdot \bI_{|\tf|< c_n},\,\, \text{i.o.}) = 0.
\]
Combining the above facts, we have show that
$$
\frac{1}{\sqrt{ nH( c_n) }} \sum_{k=1}^n \tf\circ F^k
-\frac{1}{\sqrt{  nH( c_n)}} \sum_{k=1}^n (\tf\cdot (\bI_{|\tf|< c_n}))\circ F^k \to 0$$
convergence to zero in distribution, as $n\to\infty$.

We define $\tf_{n,k}=(\tf\cdot\bI_{|\tf|<c_n})\circ F^k$ for any $n\geq 1$ and $k=0,\cdots, n$. Above analysis shows that it is enough to consider the central limit theorem for this triangle array.

We consider the stochastic process $$ X_{n,k}:=\EE_k(\tf_{n,k})-\mu_M(\tf_{n,k})$$
for $n\geq 0$.

 We first show that for any fixed $n\geq 1$, $\{X_{n,k}, k=0,\cdot,n\}$ is a stationary sequence.

\begin{lemma}\label{lemma8} For $n\geq 1$ and any random variable $h\in L_1(M,\mu)$,
\beq\label{stationary}
\EE(h\circ F^{n}|\cF_n)=\EE(h|\cF_0)\circ F^n.
\eeq
In particular, for any $n\geq 1$,  $\{X_{n,k}, k=0,\cdots, n\}$ defines a stationary process.
\end{lemma}
\begin{proof}
For $n \ge 1$, any $g\in \cF_{n}$, and $h\in L_1(M,\mu)$, the invariance of $\mu$ and the fact that $g\circ F^{-n}\in \cF_0$
implies that
\begin{align*}
\EE( \EE(h|\cF_0)\circ F^{n}\cdot g)&=\EE( \EE(h|\cF_0)\cdot g\circ F^{-n}) \\
&=\EE(\EE(h \cdot g\circ F^{-n} |\cF_0))\\
&=\EE(h \cdot g\circ F^{-n}) \\
&=\EE(h\circ F^{n}\cdot g).
\end{align*}
Thus, by the definition of conditional expectation, we have shown that the expectation of $h\circ F^{n}$ with respect to $\cF_n$ is $\EE(h\circ F|\cF_0)\circ F^{n}$, which verifies (\ref{stationary}).


Applying the above formula for $h=\tf_{n,0}-\mu(\tf_{n,0})$, we then obtain \[X_{n,k}=\EE(h\circ F^{k}|\cF_k)=\EE(h|\cF_0)\circ F^k=X_{n,0}\circ F^k.\]
\end{proof}

 Now we have shown that $\{X_{n,k}, k=0,\cdots, n\}$ is a sequence of identically distributed, stationary random variables.
Note that
\beq\label{sumg}\tf_{n,0}+\cdots+\tf_{n,0}\circ F^k-k\mu(\tf)=X_{n,0}+\cdots+X_{n,k}+\xi_{n,0}+\cdots+\xi_{n,k},\eeq
where
$$\xi_{n,k} =\tf_{n,k} - \EE_k(\tf_{n,k})=\tf_{n,0}\circ F^k-\EE_0(\tf_{n,0})\circ F^k= (\tf_{n,0}-\EE_0(\tf_{n,0}))\circ F^k.$$

Note that $\EE_0(\tf_{n,0})$ is the conditional average of $\tf_{n,0}$ on each unstable manifold; i.e., for any $ W^u \in \cW^u$, there exists $x_u\in W^u$ such that for any $x\in W^u$, $\EE_0(\tf_{n,0})(x)=\tf_{n,0}(x_u)$. Consequently, $\xi_{n,0}$ can be represented as
$$\xi_{n,0}(x)=\tf_{n,0}(x)-\tf_{n,0}(x_u),$$ with $x_u\in W^u(x)$ for any $x\in M$ such that $W^u(x)$ exists.

We now claim that
\[\frac{1}{\sqrt{(1-\theta^2) nH( c_n)}}\sum_{i=0}^{n-1} \xi_{n,i}\to 0
\]
in distribution. We use assumption (\textbf{A3}), which implies that
\[
\var\left(\sum_{i=0}^{n-1} \xi_i\right)=n\var(\xi_0)+2\sum_{k=1}^{n-1}(n-k)\cov(\xi_0\cdot\xi_{n,k})=\cO(n).
\]

 Thus, it follows from (\ref{sumg}) and Theorem \ref{MCLXn0} that we have
\[
\frac{1}{\sqrt{ \frac{1+\theta}{1-\theta}\cdot  nH( c_n) }} \sum_{k=1}^n \tf\circ F^k\to \N(0,1)
\]
in distribution.
 This completes the proof of Theorem  \ref{MCLXn2}.

\vspace{1cm}

Next we prove Theorem \ref{IPY}. Without loss of generality, we assume $\mu(\tf)=0$.
Let $J=[0,1]$. Let $D(J)$ be the space of right continuous real valued functions on $J$, endowed with the Skorohod $J_1$ topology.
 Let $X_{n,k}:=\mathbb{E}_n(\tf\cdot \bI_{|\tf|<c_n})\circ F^k$, with $n\geq 1, k=0,\cdots n$. We now consider a triangular array $\{\tX_{n,k}, n\geq 1, k=0,\ldots, n\}$, defined as
\[
\tX_{n,k}=\frac{X_{n,k}}{\sqrt{n H(c_n)} }.
\]
Note that $ H(c_n)=\var{X_{n,k}}$.
We define a random function
\beq\label{Wt1}
W_n(X,t)=\sum_{k=1}^{[tn]}\tX_{n,k}.
\eeq
Observe that $W_n(X,\cdot)$ is a right continuous step function, a measurable element of $D(J)$, and $W_n(X,0)=0$. The object of the invariance principle is to show that $W_n(X,\cdot)$ converges weakly to the standard Brownian motion process $W$ on $D(J)$.

Next, we consider the process $Y_{n,k}:=\tf_{n,0}\circ F^k$.  We form a random function such that
\beq\label{Wt}
W_n(Y,t)=\sum_{k=1}^{[tn]}\frac{Y_{n,k}}{\sqrt{n H(c_n)}}.
\eeq
It follows from (\ref{sumg}) that
\[
\sum_{k=1}^{[nt]}Y_{n,k}=\sum_{k=1}^{[nt]}X_{n,k}+\sum_{k=1}^{[nt]}\xi_{n,k}.
\]
Combining (\textbf{A3}) and Theorem \ref{IP}, we obtain Theorem \ref{IPY}, which states that  $W_n(Y,\cdot)\to \sqrt{ \frac{1+\theta}{1-\theta}}\cdot  W$ weakly on $D(J)$ as $n\to\infty$.

\section{Proof of Proposition \ref{Thm3}.}\label{Prop7}

For any piecewise H\"{o}lder continuous observable $f:\cM\to \mathbb{R}$, we let $\tf$ be its induced function. Assume $f\in \cH_{\gamma}$, with $\gamma\in (0,1)$, and we denote \beq\label{An}A_n=\sqrt{ nH_{\tf}(c_{n})}\eeq  where
$H_{\tf}(t):=\EE(\tf^2\bI_{ |\tf|<t})$.
And $S_n(\tf,t)=\sum_{k=1}^{[nt]}(\tf\circ F^k-\mu(\tf))$, then $W_n(\tf,t):=S_n(\tf,t)/A_n$. Moreover, we denote $\cS_n(f,t)=\sum_{k=1}^{[nt]}(f\circ T^k-\mu_{\cM}(f))$.
Theorem \ref{IPY} implies that $W_n(\tf,t)$ converges to the Brownian motion $W$ in $D(J)$. Taking a random time $N(t)$, we consider whether or not the same limit holds for $W_N(\tf,t)$.


To prove Proposition \ref{Thm3}, we need the following lemma from Billingsley \cite{Bil99}-Theorem 14.4 about random change of time. Let $N$ be a stopping time. We define $W_N(\tf,t)$ by
\[
W_N(\tf,t)(\omega)=\frac{S_N(\tf,t)}{A_N}(\omega)=\frac{1}{\sqrt{N(\omega) H_{\tf}(c_{N})}}\sum_{k=1}^{[N(\omega)t]}(\tf\circ F^k(\omega)-\mu(\tf)).
\]
\begin{lemma}\label{Nnstop} If $N_n$ is a random index with $N_n/n\to a$ in probability, where $a>0$ is a positive constant, then $W_n(\tf,t)\to \sigma\cdot W(t)$ for some $\sigma>0$ implies that $W_{N_n}(\tf,t)\to  \sigma\cdot  W(t)$, as $n\to\infty$.
\end{lemma}
We now prove Proposition \ref{Thm3}.

Define a measure $\nu$ on $M$ such that $\frac{d\nu}{d\mu}(x)=  m\mu_{\cM}(M)$ for any $x\in M_m$. Clearly, (\ref{WNtft}) holds with respect to the measure $\nu$.
Given $n> 1$,  we fix $n' = [n/\mu(\R)]$.
Let
\[n''(x)=1+ \sum_{j=0}^{[n]} \bI_{M}(f\circ T^j(x)),\,\,\,\,\,\forall n\geq 1,
\]
then $n''$ is the number of returns within $n$ iterations under $T$, and $n''$ is a stopping time with respect to the filtration $\{\cF_n\}$. Thus,  by Lemma \ref{Nnstop}, we have
\beq\label{WNtft}
W_{n''}(\tf,t)\to \sqrt{\frac{1+\theta}{1-\theta}}\cdot  W(t)\eeq in $D(J)$.

  Moreover it implies that $S_{n''}(\R,t)\leq nt\leq S_{n''}(\R,t)$.   Let $\hat A_{n}=\sqrt{nH_{R}(c_n)}$.

Apply Lemma \ref{Nnstop} to $\R$ for the stopping time $n''$, we know that for $n\to \infty$, $$\frac{S_{n''} (\R,t)}{\hat A_{n''}}\to \sqrt{\frac{1+\theta}{1-\theta}}\cdot  W(t).$$
    This implies that $\frac{(n'-n'')\mu(\R)}{\hat A_{n''}}\to  \sqrt{\frac{1+\theta}{1-\theta}}\cdot  W(t)$ as $n\to\infty$.
Thus  $\{\frac{(n'-n'')\mu(\R)}{\hat A_{n''}}\}$ is  tight, which implies that for any $\varepsilon>0$, there exists $C>0$ such that
$$\nu(|n'-n''|\leq C A_{n''})\geq 1-\varepsilon.
$$

   Note that for $n$ large, $n'>n''$, thus as $n\to\infty$,
    $$I_{1,n}:=\frac{S_{n'} ( \tf,t)-S_{n''}  (\tf,t)}{A_n}=\frac{S_{n'-n''} (\tf,t)}{A_n}=\frac{S_{n'-n''} (\tf,t)}{A_{n' -n''}}\cdot\frac{A_{n' -n''}}{A_{n}}
    \to 0. $$
Next, we have
$$\S_{n}( f,t)-S_{n''} ( \tf,t)\leq \|f\|_{\infty}\left(nt-S_{n''}(\R,t)\right).$$
    Note that $S_{n''}(\R,t)\leq nt$, thus $[nt]-S_{n''}(\R,t)=k$ implies that $T^{[nt]}(x)\in T^k(M_m)$ for some $m>k$.
    Then
        $$\mu(([nt]-S_{n''}(\R,t))=k)\leq \sum_{m\geq k}\mu
        (M_m)\leq Ck^{-2}.$$
As a result the following process also converges to zero in probability as $n\to\infty$:
\beq\label{I2n}I_{2,n}:=\frac{\S_{n} (f,t)-S_{n''} ( \tf,t)}{A_n}\leq \|f\|_{\infty}\frac{[nt]-S_{n''}(\R,t)}{A_n}\to 0.\eeq

 Combining the above facts as well as Theorem \ref{IPY}, we have shown the convergence holds
with respect to the measure $\nu$,
\beq\label{n'} \frac{\S_n(f,t)}{A_n} = \frac{S_{n' }(\tf,t)}{A_{n' }}\cdot\frac{A_{n' }}{A_n}+\frac{S_{n''}  (\tf,t)- S_{n' }  (\tf,t)}{A_n}+
    \frac{\S_n (f,t)-S_{n''}  (\tf,t)}{A_n}\to   \sqrt{\frac{(1+\theta)\mu_{\cM}(M) }{1-\theta}}\cdot  W(t)\eeq
as $n\to\infty$.

Next we consider  the measure $\mu_{\cM}$. Since $\cM$ can be built into a tower based on $M$ with height function $\R$, $(M,\nu)$ can be viewed as isometric to the space $(\cM,\mu_{\cM})$. For any $x\in \cM$, let $\pi: \cM\to M$ be the projection onto the base along trajectories. For any Holder function $f:\cM\to \mathbb{R}$, if $x\in \cM\setminus M$ and $\min\{k\geq 1\,:\,T^k x\in M\}=i$ then $x$ belongs to the $i$-th level of the tower. For any $n>1$, $t\in (0,1]$ and $x\in \cM$, let $n''$ be the number of returns to the base $M$ within $[nt]$ iterations along the trajectory of $x$.  Then according to (\ref{I2n}), the following sequence converges to zero in probability:
    $$\frac{\S_n (f,t)-\S_n (f\circ \pi,t)}{A_n}\leq \|f\|_{\infty} \frac{n-S_{n''}( \R
    \circ \pi,t) }{A_n}\to 0.$$
         Then  $$\frac{\S_n( f,t)}{A_n}= \frac{\S_n (f\circ \pi,t)}{A_n}+\frac{\S_n (f,t)-\
         S_n (f\circ \pi,t)}{A_n}\to   \sqrt{\mu_{\cM}(M)\frac{1+\theta}{1-\theta}}\cdot W(t)$$
         in $D(J)$ as $n\to\infty$.

\section{Sufficient conditions for estimating diffusion coefficients}
The limiting law (\ref{ClTZ}) can be interpreted in physical terms as  {\it{superdiffusion}}, and the constant factors in the denominator refer to the so-called superdiffusion constant, which plays an important role in physics. Next, we introduce a set of new conditions, under which the superdiffusion constant can be characterized more specifically.  These conditions can be applied to hyperbolic attractors. Later, we apply this theorem to all dynamical systems considered in this paper.\\

\begin{itemize}
\item[\textbf{(B1)}]  Let  $f\in \cH_{\gamma}$ with exponent $\gamma\in (0,1)$. Assume there exist $N>0$ and a set of real numbers $\cA_f=\{a_1,\cdots, a_{N}\}$,  and for each $n\geq 1$, the level set $M_n$ is decomposed into $N$ connected sets $M_n=\cup_{k=1}^{N} M_{n,k}$. Let $U_{n,k}=\cup_{m=n}^{\infty} M_{m,k}$. We now define a function
    \beq\label{Jx}
J_f=\sum_{k=1}^{N} a_k \cdot \R\cdot \bI_{U_{1,k}}.\eeq Assume there exists an observable $E=E(f)$, such that  the induced function $\tf$ satisfies:
$$\tf(x) =J_f(x) +E(x)$$
\item[(\textbf{B2})]   Assume there exists $c_{M,f}>0$  such that  \beq\label{cMRn}c_{M,f} =\lim_{n\to\infty} n^2\mu(|J_f|\geq n).\eeq
\item[(\textbf{B3})] Assume there exists $\theta\in (-1,1)$ 
such that
\beq\label{EEXXA1}\EE((J_f\cdot \bI_{|J_f|<c_n})\circ F| \cF_0) = \theta \cdot J_f\cdot \bI_{|J_f|<c_n}+\cE_{n,0},\eeq
where $\cE_{n,0}\in \cF_0$.

  \item[(\textbf{B4})]
  Suppose $\{\tg_{n,k}, n\geq 1, k=0,\cdots, n\}$, for each $n\geq 1$, each $k=0,\cdots,n$, $\tg_{n,k}$ is an induced function by some $g_{n,k}\in\cH_{\gamma}$, and $\tg_{n,k}\leq c_n$. Then for any $k=1,\cdots, 2\ln n$, we have \beq\label{eq:condition D_r(u_n)2f13}
|\mu(\tg_{n,k}\circ F^k\cdot{\tg_{n,k}})-\mu({\tg_{n,k}})^2|\leq C \theta^{k}\ln n;
\eeq
 and for
  $k=2\ln n,\cdots,n$,
\beq\label{eq:condition D_r(u_n)2f1}
|\mu(\tg_{n,k}\circ F^k\cdot{\tg_{n,k}})-\mu({\tg_{n,k}})^2|\leq C \theta^{k/2}.
\eeq
where  $\theta\in (0,1)$ is a constant.
\end{itemize}
Remark:
Note that (\textbf{B1}) and (\textbf{B4}) implies that
\beq\label{CovEnk}\var(\sum_{k=0}^{n-1} E_{n,k})<C n\eeq for some uniform constant $C>0$, where $E_{n,k}=(E\cdot \bI_{\mathcal{R}<c_n})\circ F^k$, and  $c_n=\sqrt{n\ln\ln n}$.

Furthermore, (\textbf{B3}) and (\textbf{B4}) imply that
  \beq\label{covcEn}\sum_{k=1}^{n-1}(n-k)\cov(\cE_{n,0}\cdot\cE_{n,0}\circ F^k)=\cO(n).\eeq

\begin{theorem}\label{A5con}  Let $f$ be a Holder continuous function on $\cM$ with exponent $\gamma\in (0,1)$ and $\mu_{\cM}(f)=0$. If \textbf{(B1)-(B4)}   hold, then \beq\label{ClTZft}\frac{\tf+\cdots +\tf\circ F^{n-1}}{ \sqrt{\tfrac{1+\theta}{1-\theta}\cdot c_{M,f}\cdot  n\ln n}}\xrightarrow{d} \N(0,1)\eeq
converges in distribution, as $n\to\infty$. Moreover,  the  CLT holds for $\{f\circ T^n\}$:
\beq\label{ClTZff}\frac{f+\cdots +f\circ T^{n-1}}{\sqrt{ \tfrac{1+\theta}{1-\theta}\cdot\mu_{\cM}(M) \cdot c_{M,f}\cdot n \ln n}}\xrightarrow{d} \N(0,1)\eeq
converges in distribution, as $n\to\infty$.\end{theorem}
\begin{proof}
 Assume \textbf{(B1)-(B4)}   hold. We define $X_n=J_f\circ F^n$.  Since $J_f$ is constant on each level set of $\cR$, and $\cR\circ F^n\in \cF_n$, clearly, $X_n$ is also $\cF_n$-measurable.
Thus we have the following decomposition:
$$\tf+\cdots+\tf\circ F^{n-1}=\left(X_0+\cdots+X_{n-1}\right)+\sum_{k=0}^{n-1} E\circ F^k$$
We define $E_{n,k}=(E\cdot \bI_{\mathcal{R}<c_n})\circ F^k$, for $n\geq 1$, and $k=0,\cdots, n$.
Moreover, we claim that
\[\frac{1}{\sqrt{ n\ln n}}\sum_{i=0}^{n-1} E_{n,0}\circ F^i\to 0
\]
in distribution, as $n\to\infty$.

We use assumption \textbf{(B1)} and \textbf{(B4)}, which implies that
\[
\var\left(\sum_{i=0}^{n-1} E_{n,0}\circ F^i\right)\leq Cn
\] for a uniform constant $C>0$.  This implies that
$$\frac{1}{\sqrt{ n\ln n}}\sum_{i=0}^{n-1} E_{n,0}\circ F^i\to 0$$
in $L_2$, as $n\to\infty$.

Next it suffices to verify conditions \textbf{(A1)-(A2)} for the process $\{X_n\}$.
 Note that
\begin{align*}
 \EE(J_f^2\bI_{ |J_f|<t})&=\sum_{k=1}^{N}a_k^2 \cdot \EE(\cR^2\cdot \bI_{(x\in U_{1,k}\,:\,|a_k|\cR(x)<t)})\\
 &=2\sum_{k=1}^{N}a_k^2 \cdot \sum_{s=1}^{[t/|a_k|]} s \,\mu(U_{s,k}).
 \end{align*}
 where $U_{s,k}=\cup_{m=s}^{\infty} M_{m,k}$.
Since $ \EE(J_f^2\bI_{ |J_f|<t})\to\infty$ as $t\to\infty$. By  L'hospital's rule and (\textbf{B2}), we get
\begin{align}\label{CMX0}\lim_{t\to\infty}\frac{ \EE(J_f^2\bI_{ |J_f|<t})}{\ln t}&=\lim_{t\to\infty}\frac{\sum_{k=1}^{N}2a_k^2 \cdot \int_{0}^{t/|a_k|} s \,\mu(U_{s,k})\, ds}{\ln t}\nonumber\\
&=\lim_{t\to\infty}\sum_{k=1}^{N}2t^2  \cdot \mu(U_{t/a_k,k}) \nonumber\\
&=\lim_{t\to\infty}2t^2\,\mu(|J_f|\geq t)=2c_{M,f}.\end{align}
We choose
 a positive increasing sequence $\{c_n\}$ such that $c_n=\sqrt{n\ln\ln n}$. Clearly, we have $\lim_{n\to\infty} c_n=\infty$ and the following holds:
\beq\label{choosecn2}\lim_{n\to\infty} n \mu(|J_f|>c_n)=0\,\,\,\,\,\,\text{ and }\,\,\,\,\,\,\,\lim_{n\to\infty}\frac{c_n}{\sqrt{n  \EE(J_f^2\bI_{ |J_f|<c_n})}}=0.\eeq

It follows from the above analysis that
\beq\label{HcnMncM}
\lim_{n\to\infty} \frac{ \EE(J_f^2\bI_{ |J_f|<c_n})}{\ln n}=\lim_{n\to\infty} n^2\,\mu(|J_f|\geq n)=c_{M,f}.
\eeq
This relation also tells us that  in order to estimate the variance of the variance $J_f\cdot \bI_{ |J_f|<c_n}$, it is enough to estimate the tail distribution of the random variable $\mu(|J_f|\geq n)$.  We denote $a_M:=\max\{a_k, k=1,\cdots, N\}$ as the largest value that $I_f$ can take.  Note that for any $t>0$
\beq\label{JtaM}(|J_f|>t/a_M)\supseteq (\R>t)\supseteq  (|J_f|>a_Mt)\eeq

Combining with (\ref{HcnMncM}), we know that for some constant $C>0$:
\beq\label{muJfRn}
\mu(\R\geq n)\leq C n^{-2}.
\eeq

\vspace{0.5cm}

\noindent{\textbf{Claim:}} We claim that for any $n\geq 1$, \beq\label{hatcn}  \lim_{n\to\infty}\frac{\EE(J_f^2\cdot  \bI_{ |J_f-\EE(J_f)|< c_n})}{\ln n}=\lim_{n\to\infty}\frac{\EE(J_f^2\cdot  \bI_{|J_f|< c_n})}{\ln n}.
\eeq

\vspace{0.5cm}

To prove this claim, we first get  the following  estimation for any $b\geq 0$. Note that
\begin{align*}\lim_{n\to\infty}\frac{\EE(J_f^2\cdot  \bI_{ c_n\leq  |J_f|< c_n+b})}{\ln n}&=
\lim_{n\to\infty}\frac{2}{\ln n} \int_{c_n}^{c_n+b} s \,\mu(|J_f|>s)\, ds\\
&\leq \lim_{n\to\infty}\frac{2 (c_n+b)}{\ln n}   \mu(c_n<|J_f|<c_n+b) \\
&\leq \lim_{n\to\infty}\frac{2 (c_n+b)}{\ln n}   \mu(|J_f|>c_n) =0,\end{align*}
where we have used (\ref{HcnMncM}) in the last step estimation.

Combining with the above facts, we have
\begin{align*}
&\lim_{n\to\infty}\frac{\EE(J_f^2\cdot \bI_{|J_f-\EE(J_f)|<c_n})}{\ln n}=
\lim_{n\to\infty}\frac{\EE(J_f^2\cdot  \bI_{ -c_n+\EE(J_f)<
J_f<c_n+\EE(J_f)})}{\ln n}\\&=\lim_{n\to\infty}\frac{\EE(J_f^2\cdot  \bI_{ |J_f|< c_n})}{\ln n}+\lim_{n\to\infty}\frac{\EE(J_f^2\cdot  \bI_{ c_n\leq J_f< c_n+\EE(J_f)})}{\ln n}+\lim_{n\to\infty}\frac{\EE(J_f^2\cdot  \bI_{ -c_n+\EE(J_f)\leq J_f< -c_n})}{\ln n}\\
&=\lim_{n\to\infty}\frac{\EE(J_f^2\cdot  \bI_{ |J_f|< c_n})}{\ln n}.\end{align*}
This finishes the claim (\ref{hatcn}).\\
\\

In addition, we can check that
\begin{align*}
&\lim_{n\to\infty}\frac{\EE((J_f-\EE(J_f))^2\cdot  \bI_{ |J_f-\EE(J_f)|< c_n})}{\ln n}\\
&=\lim_{n\to\infty}\frac{\EE(J_f^2\cdot  \bI_{ |J_f-\EE(J_f)|< c_n})}{\ln n}+\lim_{n\to\infty}\frac{\EE(\EE(J_f)^2-2J_f\EE(J_f)\cdot  \bI_{ |J_f-\EE(J_f)|< c_n})}{\ln n}\\
&=\lim_{n\to\infty}\frac{\EE(J_f^2\cdot  \bI_{ |J_f|< c_n})}{\ln n}.
\end{align*}
This allows us to bypass the complication of subtracting the expectation of $J_t$ in the above limit in our  estimations.
 Thus (\textbf{B2}) implies that $$
\lim_{n\to\infty}\frac{\EE((J_f-\EE(J_f))^2\cdot  \bI_{ |J_f-\EE(J_f)|< c_n})}{\ln n}=2c_{M,f},$$ which verifies (\textbf{A1}).

Since (\textbf{B3}) and (\textbf{B4}) imply (\textbf{A2}),  we get that
\begin{align*}
\frac{\tf+\cdots+\tf\circ F^{n-1}-n\mu(\tf)}{ \sqrt{\frac{1+\theta}{1-\theta}\cdot c_{M,f} \cdot n\ln n}}\xrightarrow{d} \N(0,1)
\end{align*}
converges in distribution as $n\to\infty$ by Theorem  \ref{MCLXn2}.

 Combining with Theorem \ref{MCLXn1},   we know that  the  CLT holds for $\{f\circ T^n\}$:
\beq\label{ClTZf3}\frac{f+\cdots +f\circ T^{n-1}-n\mu_{\cM}(f)}{ \sqrt{\frac{1+\theta}{1-\theta}\cdot c_{M,f} \cdot \mu_{\cM}(M)\cdot  n (\ln n)}}\xrightarrow{d} \N(0,1)\eeq
converges in distribution, as $n\to\infty$.

\end{proof}

\section{Application to nonuniformly hyperbolic systems}
\subsection{Introduction of nonuniformly hyperbolic billiards}

Let $\cD \subset \RR^2$ be a bounded open connected domain whose boundary is a finite union of
$C^3$ compact curves:
$$\partial \cD = \Gamma_1 \cup \cdots \cup \Gamma_{l}.$$
$\cD$ is called a \emph{billiard table} and $\Gamma_1,\ldots,\Gamma_l$ are its \emph{walls}.
To generate dynamics, we let a point-like particle move inside the billiard table with unit velocity. Upon
colliding with a wall, the particle bounces instantaneously such that its angle of incidence is equal
to its angle of reflection. These dynamics are referred to as the \emph{billiard flow} on $\cD$. The billiard
flow induces a first return map $T$ to $\partial \cD$ often referred to as the \emph{billiard map}. We will
be studying the discrete-time dynamics of the billiard map and its associated statistical properties.

Assume that $\partial \cD$ has a counterclockwise orientation. By construction, we have
\[
T: \partial \cD \times [-\pi/2,\pi/2] \rightarrow \partial \cD \times [-\pi/2,\pi/2].
\]
The coordinates of the billiard map are given by $(r,\varphi)$, where $r$ is an arc length parameter
on the boundary of the billiard table and $\varphi$ is the angle of reflection relative to the
normal direction. It is known that the billiard map preserves a probability measure on the collision space
$\M = \partial \cD \times [-\pi/2,\pi/2]$ given by
$$d\mu_{\cM} = c_{\mu}\cos\varphi \, dr \, d\varphi,$$
where $c_{\mu}=\left(2\,|\partial \cD|\right)^{-1} $ is the normalizing constant.

The dynamics of the billiard map are completely determined by the shape of the table. In rectangles and
ellipses, for instance, the dynamics are completely integrable. Sinai introduced the first class of chaotic
billiards in 1970 \cite{Sin}. In fact he showed that if $\partial \cD$ is convex inwards and has no
cusps, then the system is hyperbolic, ergodic, mixing, and K-mixing. Since then, many other classes of
chaotic billiards have been studied; see for example the works of Bunimovich
\cite{Bu74,Bu79},  Markarian \cite{Ma88}, and
Wojtkowski \cite{Woj86}.
We will focus on billiards which are nonuniformly hyperbolic. Although the central limit theorem
has previously been proved in some of the cases we investigate, we believe that our methods have two
advantages: they can be applied to a wide variety of dynamical systems and will give us a strong
understanding of the variance of the normal distribution in the central limit theorem. To achieve this, we
will be utilizing the theorems developed above which rely largely on the application
of the martingale central limit theorem to the problem at hand.

The billiards in this chapter are nonuniformly hyperbolic and have
polynomial rates of mixing. This slow mixing rate causes the classical central limit theorem to fail. It is therefore
advantageous to construct a map induced by the billiard map which  enjoys an exponential decay
of correlations. In certain billiards this can be accomplished by considering a subset $M \subset \M$ in
which we ignore ``nonessential collisions;"  these are sets in the phase space which
contribute to the nonuniformity of the hyperbolicity. The classification of these collisions depends on
the billiard table we are considering, so we will leave the specifics for subsequent sections.
We can
 define a return time function $\R: M \rightarrow \NN$ by
\[
\R(x) = \min\{m \ge 1 : T^m x \in M\}
\]
and an induced billiard map $F : M \rightarrow M$ by
\[
F(x) = T^{\R(x)}(x).
\]
Note that $\R$ can be  extended to $\cM$ almost surely as defined in (1).
Furthermore, we define $m$-cells $M_m$  in $M$ as
$$M_m = \{x \in M : \R(x) = m\}.$$
The induced dynamical system preserves the probability measure $\mu$ on $M$, where
$\mu(A) = \mu_{\cM}(A)/\mu_{\cM}(M)$ for any $A \in M$.

Note that the collection of $m$-cells $\{M_m\}_{m=1}^{\infty}$ is an infinite partition of $M$ into disjoint sets, each
with positive probability. We define $\cF_0=\sigma(\cR)$ to be the $\sigma$-algebra generated by $\cR$, and $\cF_n=\sigma(\cR\circ F^m, -\infty<m\leq n)$ to be the smallest $\sigma$-algebra generated by $\cR\circ F^m$ with $m\leq n$.

To determine the diffusion constant $\tilde\sigma_f$  specifically, the following subsections will exhibit nonclassical central limit theorems for $\tf$, associated with any H\"{o}lder continuous function $f:\cM\to\mathbb{R}$ on various billiard
phase spaces; i.e., we will show that
$$\lim_{n\rightarrow \infty} \frac{1}{\sqrt{n\ln n}} \sum\limits_{i=0}^{n-1} \left[ \tf \circ F^i - n\mu_{M}(\tf) \right]
	\xrightarrow{d} \N(0,\tilde\sigma_f^2),$$
where $\tilde\sigma_f^2$ is a constant which depends on the shape of the table being studied. The main reason that we concentrate on $\tf$ is that the CLT for the process $\{f\circ T^n\}$ follows from that of $\{\tf\circ F^n\}$, as indicated by Proposition \ref{Thm3}.

This is an amazing result on the tables we study for several reasons. It is possible for
trajectories of the billiard to become stuck
in arbitrarily long sequences of ``nonessential collisions," that is, many iterations may occur in $\M \setminus M$.
Consequently, the return time map $\R$ is unbounded in these systems.
Furthermore, $m$-cells in these billiards have measure
$\mu(M_m) \asymp m^{-3}$. This means that $\R$ has infinite variance:
\[\mu(\R^2) = \sum\limits_{m=1}^{\infty} m^2 \mu(M_m) \asymp \sum\limits_{m=1}^{\infty} m^{-1}.\]

Although a trajectory (under $T$) may become
stuck in an arbitrarily large number of nonessential collisions, the central limit theorem indicates that this
is highly atypical. However, this possibility does contribute to the nonstandard scaling factor found in the theorem.
We see that the extra $\sqrt{\ln n}$ leads to the variance of the average being $\sigma_R^2 n^{-1}\ln n$ as opposed
to the classical $\sigma_R^2 n^{-1}$. Clearly, the variance in our case converges to zero more slowly, meaning that
the convergence of the time average of $\R$ to its space average is also slower.

We next review the concept of a standard pair and state a growth lemma. For an unstable curve $W$ and a probability measure $\nu_0$ on the Borel $\sigma$-algebra of $W$, we say that the pair $(W,\nu_0)$ is a \textit{standard pair} if $\nu_0$ is absolutely continuous with respect to the Lebesgue measure, $ m_W$, induced by the curve length, with density function $f(x):=d\nu_0/dm_W$ satisfying
\beq\label{lnholder}
|\ln f(x)-\ln f(y)|\leq C {d}_W(x,y)^{\gamma_0}.
\eeq
 Here the $\gamma_0$ exponent is the H\"{o}lder exponent which appears in the distortion bound for the map $F$. Also, ${d}_W(x,y)$ is the distance between $x$ and $y$ measured along the smooth curve $W$.

The notion of a standard pair was studied  by Chernov and Dolgopyat in \cite{CD}. In particular, they considered families of standard pairs
 $\cG=\{(W_{\alpha}, \nu_\alpha)\,:\, \alpha\in \cA\}$ where $\cA\subset [0,1]$.
Let $\cW=\{ W_{\alpha}\,|\, \alpha\in \cA\}$. We call $\cG$ a {\em standard family} if $\cW$ is a measurable foliation of a  measurable subset of $M$,
and there exists a finite Borel measure $\lambda_{\cG}$ on $\cA$, which defines a measure $\nu$
 on $M$  by
\beq\label{cGnm}
   \nu(B):=\int_{\alpha\in\cA} \nu_{\alpha}(B\cap
   W_{\alpha})\,
   d\lambda_{\cG}(\alpha)\hspace{1cm}
   \eeq
for all  measurable sets $B\subset M$. In the following, we denote a standard family by $\cG=(\cW,\nu)$.

Define a function $\cZ$ on standard families, such that
for any standard family $\cG=(\cW,\nu)$,
\beq\label{cZ}
\cZ(\cG):=\frac{1}{\nu(M)}\,\int_{\cA}|W_{\alpha}|^{-1}\,\lambda_{\cG}(d\alpha). \eeq

For any unstable curve $W\in \cW$, any $x\in W$,  and any $n\geq 1$, let $W^{k}(x)$ be the smooth unstable curve  in $F^k W$ that contains $F^k x$.
We define $r_k(x)$ as the minimal distance between $F^k x$ and  the two end points of $W^u(F^k x))$.

The following Growth Lemma was proved in \cite[Lemma 6]{CZ09}.
\begin{lemma}[Growth Lemma]\label{growthlemma}
Let $\cG=(\cW, \nu)$ be a standard family such that $\cZ(\cG)<\infty$. Then for any $\eps>0$,
$$\nu(r_k(x)<\eps)\le C_0\eps \cZ(F^k\cG)\leq C_1(\vartheta^{k-1}\cZ(F\cG)+C_2)\eps$$
where $C_0>0, C_1>0, C_2>0$ and $\vartheta\in (0,1)$ are constants.
\end{lemma}

{For $\gamma\in (0,1)$, let $\cH_{\gamma}$	be the set of  bounded functions $f\in L_{\infty}(M,\mu)$ for which  there exists an integer $n_0\geq 1$, such that for any connected component $A\in \sigma(\cR\circ F^{n_0})$ and any $x, y\in A$,}
\beq \label{DHC-} 	|f(x) - f(y)| \leq \|f\|_{\gamma} \dist(x,y)^{\gamma},\eeq
with
$$\|f\|_{\gamma}:= \sup_{A\in \sigma(\cR\circ F^{n_0})}\sup_{ x, y\in A}\frac{|f(x)-f(y)|}{\dist(x,y)^{\gamma}}<\infty.$$
Here  $\sigma(\cR\circ F^{n_0})$ is the $\sigma$-algebra generated by the random variable $\cR\circ F^{n_0}$.
For every $f\in \cH_{\gamma}$we define
\beq \label{defCgamma1}
   \|f\|_{C^{\gamma}}:=\|f\|_{\infty}+\|f\|_{\gamma}.
\eeq

It was shown in \cite[Theorem 3]{CZ09} that for  $f, g\in \cH_{\gamma}$, and any integer $k$, the correlations of $f$ and $g\circ   F^k$ satisfy:
\beq\label{correT}\text{Cov}(f,g\circ   F^k):=|\mu(f\cdot g\circ   F^k)-\mu(f)\mu(g)|\leq C \|f\|_{C^{\gamma}} \|g\|_{C^{\gamma}}\vartheta^k\eeq
where $C>0$ and $\vartheta\in (0,1)$ are constants.

For a fixed large constant $C_p>0$, any standard family $\cG$ with $\cZ(\cG)<C_p$ will be called a {\it{proper}} family. The following was proved in \cite[Theorem 2]{CZ09}.

 \begin{lemma}[Equidistribution]\label{growthlemma2}
If $\cG=(\cW,\nu)$ is a proper family, then for any $g\in \cH_{\gamma}$, with $\gamma\in (0,1)$,
\beq\label{properg}
|F^m\nu(g)-\mu(g)|\leq C \|g\|_{C^{\gamma}}\theta^m
\eeq
where $C>0$ is a constant.
\end{lemma}

We note here that in the sections that follow we are presenting examples of applications that demonstrate that our
main theorems hold, since in most cases the central limit theorem has been proved for the following billiards.
This was done by B\'alint and Gou\"ezel for stadia \cite{GB}, by B\'alint, Chernov, and Dolgopyat in the
case of dispersing billiards with cusps \cite{BCD}, and by Sz\'asz and
Varj\'u in the case of Lorentz gas with infinite
horizon \cite{SV07}. We believe that our method is applicable to a wide variety of systems and
will be useful in determining relevant variances and diffusion coefficients in those systems; we
intend to demonstrate this claim in subsequent sections.

\subsection{Maps with linear spreading property}
In this subsection, we consider three types of systems which generate stochastic processes with the linear spreading property as defined below.

 We say that a dynamical system has the \textbf{linear spreading property} for the one-step transition if the following three conditions are true for large enough  $m,n$:

\noindent(\textbf{L1}) Assume that there exists $\beta>1$ such that $FM_m$ only intersects  those cells $M_n$
with index $n\in \B_m\colon= [m/\beta+c_1, \beta m +c_2]$ for some constants
$c_1, c_2>0$. We also assume that
\beq\label{thetabeta}
\theta:=\frac{2\ln\beta}{\beta-\beta^{-1}}<1;\eeq\\
(\textbf{L2}) Assume there exist $c_M>0$  such that the measure  $$\lim_{m\to\infty}  m^2\mu(\cR(x)> m)=c_M;$$

\noindent(\textbf{L3}) Assume the cell $M_m$ has length $\sim m^{-1}$ and width $\sim m^{_2}$, then transition probability from $M_m$ to $M_n$
satisfies
\beq\label{Markov}p_{m,n}\colon =\mu(M_n|F(M_m))=\frac{c_0m}{n^2}+ c(m,n)m^{-2},\eeq
for any $i\geq 1$,
where $c_0=\left[\beta-\beta^{-1}\right]^{-1}$ is the normalizing constant and $c(m,n)$ is uniformly bounded.  \\

\noindent(\textbf{L4}) Assume for $k=1,\cdots, n-1$, for $x$ belongs to any unstable manifold $W^u\subset M_n$, the unstable manifolds $T^kW^u$ at $T^k x$ are expanded under $T$ by a factor $1+\lambda_k$, with
\beq\label{lambda1}\lambda_k= \frac{C}{k}+o(n^{-1}).\eeq

We first prove the H\"{o}lder continuity of $\tf$.
We now estimate the H\"{o}lder norm of $\tf$, for any $f\in \cH_{\gamma}$.
\begin{lemma}\label{Holdertf}
For any $\gamma\in (0,1)$, any $f\in \cH_{\gamma}$, any $n\geq 1$, any $x,y\in M_n$, the induced function $\tf$ has H\"{o}lder norm satisfies the following condition:
\beq\label{holdertf}
|\tf(x)-\tf(y)|\leq C\|f\|_{\gamma}  n^{1+\gamma}d(x,y)^{\gamma}
\eeq
where $C=C(\gamma)>0$ is a constant.
\end{lemma}
\begin{proof} For any $n\geq 1$, any $x,y\in M_n$,
$$|\tf(x)-\tf(y)|\leq \sum_{k=0}^{n-1}\|f\|_{\gamma} d(T^k x, T^k y)^{\gamma}$$
The images $\{T^k(M_n), k=1,\cdots, n-1\}$  keep stretching in the unstable
direction and shrinking in the stable direction, as $k$ increases, thus we can assume that $x, y$ lie on one unstable curve $W\subset M_n$. By (\textbf{L4}), for $k=1,\cdots, n-1$, the unstable manifolds $T^kW$ at $T^k x$ are expanded under $T$ by a factor $1+\lambda_k$, with
$\lambda_k\sim 1/k.$

We know that the length of $FM_n$ is of order $\cO(n^{-1})$. Using the fact that $FW$ is stretched in $FM_n$, thus its  length satisfies
\beq\label{|Fw|}|FW|\leq C n^{-1}\eeq
Moreover, by the distortion bound, we have for $m\in [1,n-1]$,
$$d(T^k m, T^m y)\leq C_1d(x,y)\prod_{l=1}^{k}(1+\lambda_l)$$
for some constant $C_1>0$. Combining with (\ref{lambda1}) and (\ref{|Fw|}), we get,
$$d(T^m x, T^m y)\leq C m d(x,y)$$
Combining the above facts, we have
\begin{align*}
|\tf(x)-\tf(y)|&\leq \sum_{k=0}^{n-1}\|f\|_{\gamma} d(T^kx,T^ky)^{\gamma}\\
&\leq \|f\|_{\gamma} C_1 d(x,y)^{\gamma} n^{1+\gamma}.\end{align*}
This implies that $\tf$ has H\"{o}lder norm of order $n^{1+\gamma}$.

\end{proof}

We then prove a lemma, which verifies (\textbf{B4}).

Let $\cM_{i,j}=\cup_{m=i}^j M_m$  be the union of cells with indices satisfying $1\leq i\leq j$.
Denote
\beq\nn \label{eq: function sequence}
{\tf}_{i,j}:=\tf\cdot \bI_{M_{i,j}}-\mu(\tf\cdot \bI_{\cM_{i,j}})
\eeq
\begin{proposition}[Exponential decay of correlations for $\tf_{i,j}$]\label{boundedcovtf}
Let $\tf$ be an induced function on $M$, and $\tf_{i,j}$ is as defined above for any $1\leq i\leq j< \sqrt{n\ln\ln n}$. Then
for all $k\leq 2\ln n$,
\beq\label{eq:condition D_r(u_n)2f}
|\mu({\tf}_{i,j}\circ F^k\cdot{\tf}_{i,j})-\mu({ \tf}_{i,j})^2|\leq C\ln n\cdot \theta^k.
\eeq
where  $\theta\in (0,1)$ is a constant. On the other hand, for $k\geq 2\ln n$,
\beq\label{eq:condition D_r(u_n)2f6}
|\mu({\tf}_{i,j}\circ F^k\cdot{\tf}_{i,j})-\mu({ \tf}_{i,j})^2|\leq C\theta^{k/2}.
\eeq
\end{proposition}


\begin{proof}
For any  set $M_m$ in $\cM_{i,j}$, we foliate it into unstable curves  that stretch completely from one side to the other. Let $\{W_{\beta}, \beta\in \cA, \lambda\}$ be the foliation, and $\lambda$ the factor measure defined on the index set $\cA$.  This enables us to define a standard family, denoted as $\cG_{i,j}=(\cM_{i,j},
\mu_{i,j})$, where $\mu_{i,j}:=\mu|_{\cM_{i,j}}$. 

Our first step in proving the decay of correlations is to estimate the $\cZ$ function of $F^k\cG_{i,j}$.
According to assumption (\textbf{L3}), $FM_m$ is a strip that has length $\sim m^{-1}$ and width $\sim m^{-2}$. Also, by construction, the density of $\mu_{i,j}$ is of order $1$ on $ FM_m$. Thus we obtain for $j< \infty$,

\begin{align*}
\cZ(F\cG_{i,j})&=\mu(\cM_{i,j})^{-1}\int_{\beta\in\cA}|FW_{\beta}|^{-1}\,\lambda_{\cG_{i,j}}(d\beta)\\
&\le C
\mu(\cM_{i,j})^{-1}\cdot \sum_{m=i}^{j} m^{-2}\\
&= C_1 i^2\cdot i^{-1}=C_1 i.\end{align*}

Note that for any large $m\leq l$, we have
$$\mu(F\cM_{i,j}\cap \cM_{m,l})\leq F_*\mu_{i,j}(r<\eps_{m,l})$$
where $\eps_{m,l}$ is approximately the length of the smallest cell in $\cM_{m,l}$, which is of order $m^{-2}$.
Using  Lemma \ref{growthlemma2}, we have that
\begin{align}\label{estMiMk}
\mu(F^k(\cM_{i,j})\cap \cM_{m,l})&\leq F^k_*\mu_{i,j}(r<\eps_{m,l})\nonumber\\
&\leq C'(\vartheta^{k-1}\cZ(FG_{i,j})+C'')\eps_{m,l}\mu(\cM_{i,j})\nonumber\\
&\leq C(C_1\|f\|_{\infty}\vartheta^{k-1}i+C'')m^{-2} i^{-2}
\end{align}

For any fixed large $k$, we truncate $ \tf_{i,j}$ at two extra levels, with $i\leq p<q\leq j$, which will be chosen later, i.e.
$$ \tf_{i,j}= \tf_{i,p}+ \tf_{p,q}+ \tf_{q,j}$$

The function $ \tf_{i,q}= \tf_{i,p}+ \tf_{p,q}$ is bounded with $\|\tf_{i,q}\|_{\infty}\leq C\|f\|_{\infty} q$, and H\"{o}lder norm $\| \tf_{i,q}\|_{\gamma}\leq C\|f\|_{\gamma} q^{1+\gamma}$. Thus by (\ref{correT}), we know that
\beq\label{iqiq}
\mu( \tf_{i,q}\circ F^k,  \tf_{i,q})\leq C\| \tf_{i,q}\|_{C^{\gamma}}^2 \theta^k +\mu( \tf_{i,q})^2\leq C q^{2+2\gamma}\theta^k +\cO(q^{-2})
\eeq
where we used the fact that
$$\mu( \tf_{i,q})=-\mu( \tf_{q,j})=\cO(q^{-1})$$

Next we estimate
\begin{align*}
&\mu( \tf_{i,p}, \tf_{q,j}\circ F^k)\leq C\|f\|_{\infty}^2\sum_{m=i}^p\sum_{l=q}^{j} m\cdot l\cdot \mu(M_l\cap F^k M_m)\\
&=C\|f\|_{\infty}^2\sum_{m=i}^p\sum_{l=q}^j m\sum_{s=1}^l \mu(M_l\cap F^k M_m)\\
&=C\|f\|_{\infty}^2\sum_{m=i}^p\sum_{t=1}^m
\left(\sum_{s=1}^q \sum_{l=q}^j  \mu(M_l\cap F^k M_m)+\sum_{s=q}^j \sum_{l=s}^j  \mu(M_l\cap F^k M_m)\right)\\
&=C\|f\|_{\infty}^2\sum_{m=i}^p\sum_{t=1}^m
\left(q\cdot \mu(\cM_{q,j}\cap F^k M_m)+\sum_{s=q}^j   \mu(\cM_{s,j}\cap F^k M_m)\right)\\
&=C\|f\|_{\infty}^2\sum_{t=1}^i
\left(q\cdot \mu(\cM_{q,j}\cap F^k \cM_{i,p})+\sum_{s=q}^j   \mu(\cM_{s,j}\cap F^k \cM_{i,p})\right)\\
&+C\|f\|_{\infty}^2\sum_{t=i}^p
\left(q\cdot \mu(\cM_{q,j}\cap F^k \cM_{t,p})+\sum_{s=q}^j   \mu(\cM_{s,j}\cap F^k \cM_{t,p})\right)\\
&=C\|f\|_{\infty}^2
\left( (C_1\|f\|_{\infty}\vartheta^{k-1}i+C'') q^{-1}i^{-1}+\sum_{s=q}^j (C_1\|f\|_{\infty}\vartheta^{k-1}i+C'')s^{-2} i^{-1} \right)\\
&+C\|f\|_{\infty}^2\sum_{t=i}^p
\left((C_1\|f\|_{\infty}\vartheta^{k-1}t+C'') t^{-2}q^{-1}+\sum_{s=q}^j  (C_1\|f\|_{\infty}\vartheta^{k-1}t+C'')s^{-2} t^{-2}\right)\\
\end{align*}
Then one can check that
\beq\label{lesalpha07}
\mu( \tf_{i,p}, \tf_{q,j}\circ F^k)=\cO(q^{-1}(C_1\vartheta^k \ln p+C_2i^{-1}))
\eeq

Similarly, we can show that
\beq\label{lesalpha000}
\mu( \tf_{q,j}, \tf_{i,p}\circ F^k)=\cO(q^{-1}(C_1\vartheta^k \ln j+C_2q^{-1}))
\eeq

Next, we estimate
\begin{align*}
&\mu( \tf_{p,j}, \tf_{p,j}\circ F^k)\leq C\|f\|_{\infty}^2\sum_{m=p}^j\sum_{l=p}^j m\cdot l\cdot \mu(M_l\cap F^k M_m)\\
&=C\|f\|_{\infty}^2\sum_{t=1}^p
\left(p\cdot \mu(\cM_{p,j}\cap F^k \cM_{p,j})+\sum_{s=p}^j   \mu(\cM_{s,j}\cap F^k \cM_{p,j})\right)\\
&+C\|f\|_{\infty}^2\sum_{t=p}^j
\left(p\cdot \mu(\cM_{p,j}\cap F^k \cM_{t,j})+\sum_{s=p}^j   \mu(\cM_{s,j}\cap F^k \cM_{t,j})\right)\\
&=C\|f\|_{\infty}^2
\left( (C_1\|f\|_{\infty}\vartheta^{k-1}p+C'')p^{-2} +\sum_{s=p}^j (C_1\|f\|_{\infty}\vartheta^{k-1}p+C'')s^{-2} p^{-1} \right)\\
&+C\|f\|_{\infty}^2\sum_{t=p}^j
\left((C_1\|f\|_{\infty}\vartheta^{k-1}t+C'')p^{-1} t^{-2}+\sum_{s=p}^j  (C_1\|f\|_{\infty}\vartheta^{k-1}t+C'')s^{-2} t^{-2}\right)\\
\end{align*}

Then one can check that
\beq\label{lesalpha04}
\mu( \tf_{p,j}, \tf_{p,j}\circ F^k)=\cO(p^{-1}\vartheta^k \ln j+ p^{-2})=\cO(p^{-1}\vartheta^k \ln n+ p^{-2})
\eeq

Combining the above estimations, we have
\begin{align*}
\mu( \tf_{i,j}, \tf_{i,j}\circ F^k)&\leq C_1 \vartheta^k(q^{2+2\gamma}+ p^{-1}\ln n)+C_2 p^{-1}\end{align*}

Now we choose $q=\vartheta^{-k/d}$, $p=\sqrt{q}$, where $d=6(1+\gamma)$. Then above estimations implies that
$$\mu( \tf_{i,j}, \tf_{i,j}\circ F^k)\leq C_1 \theta^k\ln n,$$
where $\theta=\vartheta^{2/3}$.
Note that for $k>2\ln n$,
$$\theta^k\ln n=\theta^{k/2} \theta^{k/2}\ln n\leq \theta^{k/2} \ln n/n^{|\ln \theta|}\leq \theta^{k/2}.$$

\end{proof}

In this subsection we will show that it follows from Theorem \ref{A5con} that the CLT holds for certain observable $f$ on dynamical systems with the linear spreading property.
\begin{proposition}\label{propRCLT} For any  observable $f$ on $\cM$, assume  $\tf\in \cH_{\gamma}$ for some $\gamma\in (0,1)$. Assume \textbf{(B1)} holds, with $J_f=I_f\cR$, such that $I_f\neq 0$ is a constant. Moreover we assume the induced map $(F,M)$ satisfies \textbf{(L1)-(L4)}. Then the  CLT holds:

\beq\label{ClTZtfl}\lim_{n\to\infty}\frac{\tf+\cdots +\tf\circ F^{n-1}-n\mu(\tf)}{\tilde\sigma_{R} I_f \sqrt{ n\ln n}}= \N(0,1)\eeq
converges in distribution, as $n\to\infty$, where $$\tilde\sigma_{R} ^2:=\lim_{n\to\infty} n\mu_{\cM}(x\in \cM\,:\, \R\geq n)=\frac{1+\theta}{1-\theta}\cdot  c_M . $$
In addition, we have
\beq\label{ClTZfl}\lim_{n\to\infty}\frac{f+\cdots +f\circ T^{n-1}-n\mu_{\cM}(f)}{\sigma_f   \sqrt{ n\ln n}}= \N(0,1)\eeq
converges in distribution, as $n\to\infty$, where \beq\label{sigmaf1}\sigma_f=  I_f\sqrt{\mu_{\cM}(M)}\cdot \tilde\sigma_R. \eeq\end{proposition}

\noindent{\textbf{Remark.}}
\begin{itemize}
\item[(1)] Note that one can extend $\R$ to the entire phase space $\cM$, by defining $$\R(x)=\min\{n\geq 1\,:\, T^n x\in M\},$$ which can be viewed as the first hitting time.  One can check that
\begin{align*}
 c_M&=\lim_{n\to\infty}n^2\mu(x\in M\,:\,\R\geq n)\\
 &=\frac{1}{\mu_{\cM}(M)}\lim_{n\to\infty}n\mu_{\cM}(x\in \cM\,:\, \R\geq n),
 \end{align*}
which implies that $$\lim_{n\to\infty}n\mu_{\cM}(x\in \cM\,:\, \R\geq n)=c_M\mu_{\cM}(M).$$ Consequently, the supperdiffusion constant $\sigma_f$ defined below as (\ref{sigmaf1}) is closely related to the tail distribution $\mu_{\cM}(x\in \cM\,:\, \R\geq n)$:
\beq\label{sigmaf2R}\sigma_f^2=\frac{1+\theta}{1-\theta}\cdot  I_f^2 \cdot \lim_{n\to\infty}n\mu_{\cM}(x\in \cM\,:\, \R\geq n). \eeq

\item[(2)] Note that if $f=1-\mu(\R)\bI_{M}$, then   $\tf=\R-\mu(\R)$, with $I_f=1$, so the diffusion constants for the processes $\{f\circ T^n\}$ and $\{\tf\circ F^n\}$ are given by:
$$\sigma_{f}^2=\frac{1+\theta}{1-\theta}\cdot  \lim_{n\to\infty}n\mu_{\cM}(x\in \cM\,:\, \R\geq n) ,\,\,\,\,\,\,  \tilde\sigma_{f}^2=\sigma_1^2/\mu_{\cM}{(M)}.$$
This implies that heavier tail for $\cR$ implies faster diffusion for these processes.
Note that $N_n:=\bI_M(x)+\cdots +\bI_M(T^{n-1}x)$ is the number of returns to $M$ within $n$ iterations under $T$, with $\mu_{\cM}(N_n)=n\mu_{\cM}(M)$. Thus the CLT for $f$ implies that
 for any $z>0$,
$$\mu_{\cM}\left(\frac{N_n-\mu_{\cM}(N_n)}{\mu_{\cM}(M)\sigma_{f_0}\sqrt{  n \ln n}}>z\right)\to \int_{z}^{+\infty} \frac{1}{\sqrt{2\pi}} e^{-\frac{x^2}{2}} dx,$$ as $n\to\infty$.

\item[(3)] Note that for any $x\in M_n$, $\frac{\tf}{\R(x)}=I_f+E(x)/n$, which implies that
$$I_f=\lim_{n\to\infty} \frac{\tf}{n}=\lim_{n\to\infty} \frac{f(x)+\cdots+f(T^{n-1}x)}{n}.$$
As $n\to\infty$, the sequence $f(x), \ldots, f(T^{n-1}x)$ accumulates to the set $\R^{-1}(\infty)$ in the phase space. Thus  $I_f$ can be viewed as the ``average value" of $f$ in the set $\R^{-1}(\infty)$.
Thus  processes $\{f\circ T^n\} $ generated by $f$  with more weight on $\R^{-1}(\infty)$ will have faster diffusion.  This property holds for all dynamical systems with the algebraic spreading property \textbf{(L1)-(L3)}. Examples include Bunimovich stadia and Skewed stadia. \end{itemize}

 Next we give the proof of  Proposition \ref{propRCLT}.

\begin{proof}
We define $J_f=I_f \cR$, with $I_f\neq 0$ being a constant.
Define $$X_n:=J_f\circ F^n-\mu(J_f),\,\,\,\,\,\,n\geq 0.$$
Next we will show that if a system has the linear spreading property, then the process $\{X_n\}$ satisfies (\textbf{B2-B3}).

We choose
 a positive increasing sequence $\{c_n\}$ such that $c_n=\sqrt{n\ln\ln n}$; then, $\lim_{n\to\infty} c_n=\infty$. Now (\textbf{B2})  follows from (\textbf{L2}):
 \begin{align*}
 \lim_{t\to\infty} \frac{\mathbb{E}(J_f^2\cdot \bI_{|J_f|<c_n})}{\ln n}&=\lim_{n\to\infty} n^2\mu(|J_f|>n)\\
&=\lim_{t\to\infty}n^2 I_f^2\mu(\R>n)=c_M I_f^2.
\end{align*}

Next, we verify condition \textbf{(B3)}.
Since $\{\R \circ F^i\}$ is stationary, it is enough to calculate the initial one-step
conditional expectations. For $m$ large, we have
\begin{align}\label{eq:linspr1}
\EE(\R\circ F(x) | \R(x) = m) &= \sum\limits_{n\in\B_m} n\cdot\left(\frac{c_0m}{n^2} + c(m,n)m^{-2}\right)\1_{M_m}(x)\nonumber\\
&=\left(\frac{2m\ln\beta}{\beta-\beta^{-1}}  + \sum_{n\in\B_m}  c(m,n)n m^{-2}\right)\1_{M_m}(x)\nonumber\\
&=(\theta m+e_m)\1_{M_m},
\end{align}
where $e_m=\sum_{n\in\B_m}  c(m,n)n m^{-2}\leq C_e$. Let $a=\mu(\R)$.
We define $X_{n,0}=(\R-a) \cdot\bI_{\cR<c_n+a})$ and
\begin{align}\label{EEcE1}
\cE_{n,0}&=\EE((\R-a) \cdot\bI_{\cR<c_n+a}) \circ F |\R)-\theta (\R-a)\cdot\bI_{\cR<c_n+a}\nonumber\\
&=\sum_{m=1}^{\beta(c_n+a)} e_m \bI_{M_m}-a\left(\EE(\bI_{\cR<c_n+a} \circ F |\R)-\theta \cdot\bI_{\cR<c_n+a}\right).\end{align}
Note that $\cE_{n,0}$ is uniformly bounded and  $\sigma(\R)$-measurable.
This verifies (\textbf{B3}).

Now the CLT follows from Theorem \ref{A5con}.
\end{proof}

Next, we study two systems that have the linear spreading property. More precisely, we will verify conditions \textbf{(B1)}, \textbf{(L1)-(L4)}, and the fact that  for any Holder observables $f$ on $\cM$,  $c_{M,f}=I_f\cdot c_M$  for these systems.

\subsubsection{Stadia} The stadium billiard table, introduced by Bunimovich in 1974 \cite{Bu74}, is comprised of two equal
semicircles which are connected by two parallel lines.
 Let $\bl>0$. We consider a region in the plane delimited by two semicircles of radius 1, joined by
two horizontal segments of length $\bl$, tangent to the semicircles. To a point on the boundary of this
set and a vector pointing inwards, we associate an image by the usual billiard reflection law. This
defines the stadium billiard map $T:\cM\to\cM$. This map admits a unique absolutely
continuous
invariant probability measure $\mu_{\cM}$.
A point in the phase space $x\in\cM$ is given by $(r,\varphi)$ where $r\in[0, 2\pi+2\bl]$ is the position on the
boundary, and $\varphi\in [-\pi/2,\pi/2]$ is the angle with respect to the normal to this boundary at $r$. The invariant measure $\mu_{\cM}$ is given by
$$d\mu_{\cM}(r,\varphi)=\frac{\cos\varphi\,dr\,d\varphi}{4(\pi+\bl)}.$$

Dynamics on the stadium have been shown to be
nonuniformly hyperbolic, ergodic, and mixing; for some discussion of these facts see
\cite{Bu74, Bu79, CM}. Chernov and Zhang proved in \cite{CZ08} that correlations in
stadia decay polynomially, in fact, they decay as $\O(1/n)$. As a consequence, billiards in stadia do not
satisfy the classical central limit
theorem.

The induced billiard map on stadia is discussed extensively in \cite{CM}. We let
$M \subset \M$ consist only of first collisions at focusing arcs and let the induced billiard map
$F: M \rightarrow M$ and return time function $\R: M \rightarrow \NN$ be defined as previously mentioned.  Note that $M$ consists of two identical parallelograms with sides bounded by $\varphi=\pm \pi/2$ and two line segment with slope $d\varphi/dr=-1/2$. So by symmetric property, the $\mu_{\cM}$ measure of $M$
\beq
\mu_{\cM}(M)=\frac{4}{4(\pi+\bl)}\int_{0}^{\pi}\int_{0}^{r/2} \cos\varphi\,d\varphi\,dr=\frac{2}{\pi+\bl}.
\eeq
Note that the measure of the set $M$ was  incorrectly calculated in Formula (3) of reference \cite{GB}, with $\mu_{\cM}(M)=\frac{\pi}{2(\pi+\bl)}$. Apparently, $\frac{\pi}{2(\pi+\bl)}$ is the measure of the two rectangles corresponding to all collision points on the two arcs.
Let $A=[0,\bl]\cup[\pi+\bl, \pi+2\bl]$ be the set of position coordinate $r$ of all collisions on the two straight sides of the stadium.

 Note that the stadium billiard has $4$ singular points based on the two arcs, each of which has infinite free flight in the unfolding space of the table (by removing the straight sides). This implies that $\{M_n\}$ has exactly $4$ converging subsequences, denoted as $M_n=\cup_{k=1}^4 M_{n,k}$. We also denote $M_{n,5}:=M_{n,1}$ for convenience of notations. For $x\in M_n$, its image will hit the flat sides of the stadia $n-1$ times before hitting another arc. Also note that the collision angle along the trajectory $T^k x$, $k=1,\cdots, n-1$, is invariant. By symmetric property of the table, one can check that the entire set $\cup_{m\geq n}\cup_{k=1}^{m-1} T^k M_m$  is squeezed between two lines with equation
 $$\sin\varphi= \frac{\bl}{2n}+C n^{-2}\,\,\,\,\,\text{ and }\,\,\,\,\,\,\sin\varphi= -\frac{\bl}{2n}+C n^{-2}$$ for some uniform constant $C>0$.
One can check that
 \begin{align*}
 \mu_{\cM}(x\in \cM\,:\, \R\geq n)&=\frac{1}{2}\sum_{m\geq n}\sum_{k=0}^{m-1} \mu_{\cM}(T^k M_m)+\cO(n^{2})\\
 &=\frac{1}{4|\partial D|}\int_{r\in A}\int_{-\frac{\bl}{2n}+C n^{-2}} ^{\frac{\bl}{2n}+C n^{-2}} \,d\sin\varphi \, dr+\cO(n^{2})=\frac{\bl^2}{4n(\pi+\bl)}+\cO(n^{2}).
 \end{align*}

 Thus
 \begin{align*}
 c_M&=\lim_{n\to\infty}n^2\mu(x\in M\,:\,\R\geq n)\\
 &=\frac{1}{\mu_{\cM}(M)}\lim_{n\to\infty}n\mu_{\cM}(x\in \cM\,:\, \R\geq n)
 =\frac{\bl^2}{4(\pi+\bl)\mu_{\cM}(M) }.
 \end{align*}
This verifies (\textbf{L2}).


 It was shown in \cite{Bu74, CM,CZ07} that $Fx\in M_n$ for any $x\in M_m$,
where $n\in \B_m := [m/3+c_1, 3 m +c_2]$ for some constants
$c_1, c_2>0$. Moreover, for condition (\textbf{L1}) we have
\beq\label{thetabeta3}
\theta:=\frac{3\ln 3}{4}=0.824<1.\eeq
   It was shown in \cite{BSC90,BSC91} that if $x \in M_m$, then $F x \in M_k$ for some
$m/3 + o(1) \le k \le 3m + o(1)$ and that we have the transition probability
$$\mu (F x \in M_k | x \in M_m) = \frac{3m}{8k^2} + \O\left(\frac{1}{m^2}\right)$$
which verifies (\textbf{L3}). It was proved in \cite{CM}~ Chapter 8, that  for $k=1,\cdots, n-1$, for $x$ belongs to any unstable manifold $W^u\subset M_n$, the unstable manifolds $T^kW^u$ at $T^k x$ are expanded under $T$ by a factor $1+\lambda_k$, with
$\lambda_k= \frac{C}{k}+o(n^{-1})$. This verifies (\textbf{L4}).

 Let $f$ be a piecewise H\"{o}lder continuous on $\cM$ with H\"{o}lder exponent $\gamma\in (0,1)$ and we assume $\mu_{\cM}(f)=0$. Here we also assume $f$ is H\"{o}lder continuous  on a small neighborhood of the set $\{(r,0)\,:\, r\in A\}$.
 Next define $I_f$ such that
$$I_f= \frac{1}{2\bl}\int_{r\in A} f(r,0)\,  dr.$$
 Note that for any fixed $k=1,\ldots, 4$, any $x\in M_{n,k}$, its forward images $T^i x=(r_i, \varphi_i)$ will only hit uniformly on the flat sides, for $i=1,\ldots, n-1$, with $$\varphi_i=\varphi_1=\frac{2\bl}{2n}+\cO(n^{-2})$$  and $$|r_i-r_{i-1}|=\frac{2\bl}{n}+\cO(n^{-2}).$$
Since the Riemann integral approximation of $I_f$ can be obtained by sampling $f(r,\varphi_1)$   with respect to the uniform partition of each interval $r\in [0,\bl]$ and $r\in [\bl+\pi, \pi+2\bl]$; using the H\"{o}lder continuity of $f$,
 one can show that for $n$ large, any $x\in M_{n,i}$, we denote $f(T^k x)=(r_k,\varphi_1)$, then
 $$|f(r_k, 0)-f(r_k,\varphi_1)|\leq C\|f\|_{\gamma} n^{-\gamma} $$
 This implies that
\begin{align*}
|1_f &-\frac{1}{n} \sum_{k=0}^{n-1} f\circ T^k(x)| \leq C\bl^{\gamma}\|f\|_{\gamma} n^{-\gamma},
\end{align*}
where $\gamma\in (0,1)$ is the H\"{o}lder exponent of $f$.  Note that $I_f$ does not depend on $k$.
This implies that we can define $J_f=I_f \cR$. Moreover, there exists a function $E=E(f)$, such that for any $x\in M_n$,
$$\tf(x)=\sum_{k=0}^{n-1} f\circ T^k(x)=J_f (x) +E(x),$$ where $E=\cO(\cR^{1-\gamma})$. We define $E_{n,k}=(E\cdot \bI_{\cR<c_n})\circ T^k$, for $n\geq 1$, and $k=0,\ldots, n-1$.
Note that $$E_{n,k}\leq C c_n^{\gamma}=C (n\ln\ln n)^{\frac{\gamma}{2}}.$$

This verifies (\textbf{B1}).

 Now  it follows that (\ref{ClTZf}) holds with
\beq   \frac{1+\theta}{1-\theta}=\frac{4+3\ln 3}{4-3\ln 3}. \eeq

Therefore, the stadium satisfies conditions \textbf{(L1)-(L3)}
and all assumptions of Proposition \ref{propRCLT} are satisfied.

Combining the above facts and Proposition \ref{propRCLT}, we have proved the following results.
\begin{theorem} Let $f$ be a piecewise  H\"{o}lder observable on $\cM$ with $\mu_{\cM}(f)=0$ and $\tf\in \cH_{\gamma}$, such that  $f$ is H\"{o}lder continuous on a small neighborhood of the singular set $\{(r,0)\,:\, r\in A\}$. Assume $\int_{r\in A} f(r,0)\,  dr\neq 0$. Then the  sequence
\beq\frac{\tf+\cdots +\tf\circ F^{n-1}}{ \tilde\sigma_f \sqrt{  n \ln n}}\xrightarrow{d} \N(0,1)\eeq
converges in distribution, as $n\to\infty$, with
$$\tilde\sigma_f^2=\frac{4+3\ln 3}{4-3\ln 3}\cdot\frac{( \int_{r\in A} f(r,0)\,  dr)^2 }{32}.$$
Moreover,
\beq\frac{f+\cdots +f\circ T^{n-1}}{ \sigma_f \sqrt{  n \ln n}}\xrightarrow{d} \N(0,1)\eeq
converges in distribution, as $n\to\infty$, with
\beq\label{sigmafstadia}\sigma_f^2=\frac{4+3\ln 3}{4-3\ln 3}\cdot\frac{( \int_{r\in A} f(r,0)\,  dr)^2 }{16(\pi+\bl)}.\eeq

\end{theorem}

\subsubsection{Skewed stadia}

We now turn our attention to the skewed (or drivebelt) stadia. These tables are constructed by connecting a major arc $\Gamma_1$ with central angle $\theta_0\in (\pi, 3\pi/2)$ and a minor arc $\Gamma_2$ with central angle $\theta_1\in (0, \pi/2)$ by two straight lines of length $\bl$. We assume both arcs have radius $1$. These billiards were introduced by Bunimovich in
\cite{Bu74},
where he also established their hyperbolicity and ergodicity.
More recently, \cite{CZ08} proved that skewed stadia have polynomial decay of correlations. We will use facts from both references in our analysis.

By unfolding the table, we can first consider the phase space $\M$ to be made up only of collisions with
the arcs.
Unlike straight stadia, using this method produces billiards
with finite horizon since the lines used to connect arcs are not parallel.

The billiard map $T:\cM\to\cM$ admits a unique absolutely
continuous
invariant probability measure $\mu_{\cM}$.
A point in the phase space $x\in\cM$ is given by $(r,\varphi)$ where $r\in[0, \bl_0]$ is the position on the
boundary, with $|\partial\cD|=\theta_0+\theta_1+2\bl$, and $\varphi\in [-\pi/2,\pi/2]$ is the angle with respect to the normal to this boundary at $r$. The invariant measure $\mu_{\cM}$ is given by
$$d\mu_{\cM}(r,\varphi)=\frac{\cos\varphi\,dr\,d\varphi}{2|\partial\cD|}.$$

We let $M$ be the set of all
first collisions with a given arc, so that $M_m$ is made up of points which collide with the same arc $m$ times.
In \cite{CZ05a},  the induced space $M$ and the
structure of the associated $m$-cells described in detail.
We let the induced billiard map
$F: M \rightarrow M$ and return time function $\R: M \rightarrow \NN$ be defined as previously mentioned.  Note that $M$ is the union of two parallelograms,
one corresponds to the smaller arc, the other to the larger arc, we call them
{\it{small}} and {\it{big}} parallelograms, respectively. The large  parallelogram has sides bounded by $\varphi=\pm \pi/2$ and two line segment with slope $d\varphi/dr=-1/2$; and the small  parallelogram has sides bounded by $\varphi=\pm \pi/2$ and two line segment with slope $d\varphi/dr=-1/2$.
Thus the measure of $M$ satisfies
\begin{align*}
\mu_{\cM}(M)&=\frac{1}{2 |\partial\cD|}\int_0^{\theta_0}\int_{(\theta_0-\pi-r)/2}^{(\pi-r)/2} \cos\varphi\,d\varphi\,dr+\frac{1}{2 |\partial\cD|}\int_0^{\theta_1}\int_{(\theta_1-\pi-r)/2}^{(\pi-r)/2} \cos\varphi\,d\varphi\,dr\\
&=\frac{1}{2|\partial\cD|}(\sin\theta_0+\sin\theta_1).
\end{align*}

We define the $r$-coordinate such that $r=0$ at one end point of the major arc, and $r=\theta_0$ at the other end point. Let $A=[0,\theta_0-\pi]\cup [\pi, \theta_0]$. Note that there exist two symmetric families of vectors $A_1:=\{(r,0)\,:\, r\in [0,\theta_0-\pi]\}$ and $A_2:=\{(r,0)\,:\, r\in [\pi, \theta_0]\}$, such that $T A_1=A_2$ and $T A_2=A_1$. Moreover, any $x\in A_1\cup A_2$, its trajectory passes through the center of the major arc, and reflects from one-side of $\Gamma_1$ to the other side.

Note that Skewed stadia is similar to the Bunimovich Stadia, in particularly, the verification of \textbf{(L3)-(L4)} are the same, so we will only concentrate on \textbf{(L1)-(L2)}.

 To estimate the constant $c_M$ in (\textbf{L2}), we need to estimate the measure of the set $(\R\geq n)$. Note that for $x\in M_n$, its image will hit the set $A_1\cup A_2$ for $n-1$ times before exiting to other part of the boundary. Also note that the collision angle along the trajectory $T^k x$, $k=1,\ldots, n-1$, is invariant. Thus the set $\cup_{k=1}^{n-1} T^k M_n$  is squeezed between by a line with equation
 $$\sin\varphi= \frac{\theta_0-\pi}{n}+C n^{-2}$$ for some constant $C>0$ and the line $\varphi=0$.  Thus
 $$\mu_{\cM}(x\in \cM\,:\, \R\geq n)=\frac{1}{4 |\partial\cD|}\int_{r\in A}\left( \frac{\theta_0-\pi}{n}+C n^{-2}\right)\, dr=\frac{(\theta_0-\pi)^2}{ 2n|\partial\cD|}+\cO(n^{-2}).$$
 This implies that
 \begin{align*}
 c_M&=\lim_{n\to\infty}n^2\mu(x\in M\,:\,\R\geq n)\\
 &=\frac{1}{\mu_{\cM}(M)}\lim_{n\to\infty}n\mu_{\cM}(x\in \cM\,:\, \R\geq n)
 =\frac{(\theta_0-\pi)^2}{ 2|\partial\cD|\mu_{\cM}(M) }.
 \end{align*}

The skewed stadium enjoys many of the same properties as the straight stadium; for instance, the induced billiard
map $F$ has exponential decay of correlations
and the measure of an $m$ cell is $\mu(M_m) \asymp m^{-3}$.
However, for a point $x\in M_m$ we have that $Fx \in M_n$, where
$$\frac{1}{7}m - \O(1) \le n \le 7m + \O(1).$$
Note that there are $4$ families of subsequences in $M_n'\subset M_n$, $n\geq 1$ that accumulate at the tangential vectors of the two arcs. These subsequences have smaller order of measure $\mu(M_n')=\cO(n^{-4})$. Consequently, we will ignore these points in below estimations as one can check that they do not contribute to abnormal central limit theory.

Although points still spread linearly in $m$, we see that they  can travel further than those in straight
stadia. This wider range affects the normalizing constant present in the transition probabilities between cells,
and we have
\beq\label{eq:sstransprob}
\mu(Fx \in M_n | x \in M_m) = \frac{7m}{48n^2} + \O\left(\frac{1}{m^2}\right).
\eeq
Therefore, skewed stadia are linear spreading, and we have the CLT (\ref{ClTZf})
with
\beq \sigma^2 =  \frac{24+7\ln 7}{24-7\ln 7}.\eeq
This verifies (\textbf{L1}).

For skewed stadia, we define the set $A_1\cup A_2\subset\{x\in \cM\,:\, \R(x)=\infty\}$ consist of the family of two-periodic points whose trajectories passing through the diameter of the major arc.  In addition, we define $A=\{r :\, (r,0)\in A_1\cup A_2\}$ as the range of $r$-coordinates of these periodic points.   For any piecewise  H\"{o}lder observable $f$ on $\cM$ with $\mu_{\cM}(f)=0$ and $\tf\in \cH_{\gamma}$, such that $\int_{r\in A} f(r,0)\,  dr\neq 0$ and $f$ is H\"{o}lder continuous on a small neighborhood of the singular set $\{(r,0)\,:\, r\in A\}$.

We define $I_f$ such that
$$I_f= \frac{1}{2(\theta_0-\pi)}\int_{r\in A} f(r,0)\,  dr.$$
 Note that for any $n$ large, any $x\in M_n\setminus M_n'$, the forward images $T^k x=(r_k,\varphi_k)$ will only hit evenly on the two end sides of $\Gamma_1$, $k=1,\ldots, n-1$, accumulating to points in $A_1\cup A_2$, as $n\to\infty$. In particular, we have $$\varphi_k=\varphi_1=\frac{2|A_1|}{n}+\cO(n^{-2}).$$
Also note that the collisions at the curves $A_1$ and $A_2$ alternate, thus $$r_{k}=r_1+(k-1)\varphi_1=r_1+\frac{2(k-1)|A_1|}{n}+\cO((k-1)n^{-2})$$ for $k=1,\ldots, n-2$. Since $\varphi_k=\varphi_1$, for $k=1,\ldots n-1$, we define
$$\bar f(r_k)=f(r_k,\varphi_1).$$

Since the Riemann integral approximation of $I$ is obtained by uniform partition of each interval $[0,\theta_0-\pi]$ and $[\pi, \theta_0]$, by the Holder continuity of $f$ one can show that for $n$ large, any $x\in M_n$,
\begin{align*}
\tf(x)&=\sum_{k=0}^{n-1} f\circ T^k(x)=\sum_{k=0}^{n-1} \bar f(r_k)\\
&=\sum_{k=0}^{n-1} \bar f(r_1+\frac{2(k-1)|A_1|}{n}+\cO((k-1)n^{-2}))\\
&=\frac{n}{|A_1|+|A_2|}\int_{A_1\cup A_2} f(r,\varphi) dr +E(x),\end{align*}
 where $E=\cO(R^{1-\gamma})$. Here we used the fact  that the function $f$ is H\"{o}lder continuous with exponent
 $\gamma$ in the variables $(r,\varphi)$ on a small neighborhood of $A_1\cup A_2$. This verifies \textbf{(B1)}.
Combining the above facts and Proposition \ref{propRCLT}, we have proved the following theorem.
\begin{theorem} Let $f\in\cH_{\gamma}$  with $\mu_{\cM}(f)=0$ and $f\in \cH_{\gamma}$, such that $f$ is H\"{o}lder continuous on a small neighborhood of the singular set $\{(r,0)\,:\, r\in A\}$. Assume $I_f\neq 0$, then the  sequence
\beq\frac{f+\cdots +f\circ T^{n-1}-n\mu_{\cM}(f)}{ \sigma_f \sqrt{ \cdot n \ln n}}\xrightarrow{d} \N(0,1)\eeq
converges in distribution, as $n\to\infty$, with $$\sigma_f^2=\frac{24+7\ln 7}{24-7\ln 7}\cdot  \frac{ (\int_{r\in A} f(r,0)\,  dr)^2 }{8|\partial\cD|}.$$
\end{theorem}
\subsection{Billiards with algebraic spreading property}

 Let $f$ be a Holder continuous function on $\cM$ with exponent $\gamma\in (0,1)$. Assume (\textbf{B1}) and (\textbf{B2}) hold. More precisely, there exist $N>0$ and a set of finite real numbers $\cA_f=\{a_1,\cdots, a_{N}\}$,  and for each $n\geq 1$, the level set $M_n$ is decomposed into $N$ connected sets $M_n=\cup_{k=1}^{N} M_{n,k}$. Let $U_{n,k}=\cup_{m=n}^{\infty} M_{m,k}$. We now define a function
    \beq\label{Jx1}
J_f=\sum_{k=1}^{N} a_k \R\cdot \bI_{U_{1,k}}.\eeq Assume there exists an observable $E=E(f)$, such that the induced function $\tf$ satisfies:
$$\tf(x) =J_f(x) +E(x)$$ with the property that
$\var(\sum_{k=0}^{n-1} E_{n,k})<C n$ for some uniform constant $C>0$, where $E_{n,k}=(E\cdot \bI_{\mathcal{R}<c_n})\circ F^k$, and  $c_n=\sqrt{n\ln\ln n}$.
Moreover, (\textbf{B2}) implies that  there exists $c_M>0$  such that   $$\lim_{m\to\infty} m^2\mu(|J_f|\geq m)=c_M.$$

In this subsection, we first list some sufficient conditions (\textbf{G1})-(\textbf{G2}) for dynamical systems to guarantee the \textit{algebraic spreading property}, which is defined for $k\in [1,N]$:

\begin{itemize}
\item[{\textbf{(G1)}}] Assume that for $m$ large, any $x\in M_{m,k}$, $Fx$ must belong to $ M_{n,k}$, with index
$n\in \B_{m}\colon= [c_1 \sqrt{m}, c_2 m^2 ]$, for some constants
$c_1>0, c_2>0$.
\item[(\textbf{G2})] Assume the transition probability $p_{n,m,k}$ from $M_{n,k}$ to $M_{m,k}$ satisfies
\beq\label{Markov1} \sum_{n\in \cB_m} n p_{n,m,k}=\cO(\sqrt{m}).\eeq
\end{itemize}

Note that conditions \textbf{(G1)-(G2)} are simply designed for the semi-dispersing billiards as well as billiards with cusps, although we could make them more general.

We denote $X_n=J_f\circ F^n-\mu(J_f)$ and $X_{n,k}=(X_n\cdot |X_n|<c_n)\circ F^k$ for $k=0,\ldots, n$, and $H(c_n)=\var(X_{n,k})$, where $c_n=\sqrt{n\ln n}$. It follows from similar argument that  if the induced system  $(F,M)$ satisfies \textbf{(G1)-(G2)}, then the process $\{X_n\}$ satisfies (\textbf{A1})-(\textbf{A2}). Next we verify condition (\textbf{A2}).

 Let us first assume (\textbf{G2}) holds. We have, for any $k=1,\cdots, N$, any $m$ large,
\begin{align}\label{eq:linspr5}
\EE((a_k\R\cdot \bI_{a_k\R<c_n})\circ F(x) | \R(x) = m) &= \sum\limits_{n\in\B_m} a_k n p_{n,m,k}\1_{M_m}(x)\cdot\1_{a_k\R<c_n}(Fx)=e_m\1_{M_m},
\end{align}
where  $e_m=\cO(\sqrt{m})$. This implies that we can take $\theta=0$ and  $\cE_{n,0}=\sum_{m=1}^{c_n^2} e_m \bI_{M_m}$.

Thus, we have
$$\EE(\cE_{n,0}\cdot X_{n,0})=\sum_{k=1}^{c_n}k\cdot e_k \mu_{\cM}(M_k)<C\sum_{k=1}^{c_n}k^{3/2} k^{-3}<C_1,$$
where $C,C_1>0$ are uniform constants.
This verifies (\textbf{A2}).

For any $n\geq 1$, $\cE_{n,0}$ is bounded by $c_n^2$ and is piecewise constant on $M$. It follows from results in \cite{CZ09} that $\cE_{n,0}$ enjoys exponential decay of correlations under the induced map $F$. More precisely, there exists $\theta_1\in (0,1)$ and constant $C>0$ such that
\beq\label{exdecaycE2}\cov(\cE_{n,0}\circ F^i \cdot \cE_{n,0})\leq C c_n^4 \theta_1^i.\eeq
This implies that there exists $C_2>0$ such that
$$\sum_{i=1}^n(n-i)\cov(\cE_{n,0}\circ F^i \cdot \cE_{n,0})\leq C_2 n.$$
Combining the above facts, we have proved the following lemma.
\begin{lemma}\label{Ed2} For algebraic spreading processes satisfy \textbf{(B1)-(B2)} and \textbf{(G1)-(G2)}, condition (\textbf{B3})  holds with $\theta=0$.
\end{lemma}
Thus, it is sufficient to verify \textbf{(B1)-(B2)} and \textbf{(G1)-(G2)} to get the CLT according to Theorem \ref{A5con}.

\begin{proposition}\label{propaCLT} For any  continuous  observable $f$ on $\cM$, assume  $\tf\in \cH_{\gamma}$ for some $\gamma\in (0,1)$ satisfies \textbf{(B1)}-\textbf{(B2)}. Moreover we assume the induced map $(F,M)$ satisfies \textbf{(G1)-(G2)}. Then the  CLT holds:
\beq\label{ClTZf2}\frac{f+\cdots +f\circ T^{n-1}-n\mu_{\cM}(f)}{ \sqrt{  c_M\mu_{\cM}(M) n \ln n}}\xrightarrow{d} \N(0,1)\eeq
converges in distribution, as $n\to\infty$.\end{proposition}

Next, we study two systems that have the property of algebraic spreading.

\subsubsection{Dispersing billiards with cusps}
Here we consider billiards with cusps, where the table is made of three curves, in which two identical curves touches tangentially at a point $p$ that forms a cusp. The table is constructed such that the tangent line of the two symmetric boundary at $p$ is perpendicular to the third curve.

The decay rates of correlations for  billiards with cusps  was first studied  in \cite{CM07}. It is known that the billiard map on these
tables is hyperbolic and ergodic; however, the hyperbolicity is nonuniform. As a result, correlations
decay polynomially. In fact, it was proved in \cite{CZ08} that
the rate of mixing is $\le \text{const}\cdot n^{-1}$. Recently, it was shown
by B\'alint, Chernov, and Dolgopyat in \cite{BCD} that a central limit theorem does hold in these systems for
H\" older continuous functions;
 we  again restrict our attention to the return time map.

On these tables the dynamics are nonuniformly hyperbolic when trajectories become trapped in a cusp for a large number of iterations.
We wish to remove these collisions from our consideration, so for simplicity, we have assumed our table has only one cusp at point $p$ with arclength parameter
$r=0$. We construct the subset $M \subset \M$ by
$$M = \M \, \backslash \,  \, \{(r,\varphi)\,:\, r\in\partial \cD\setminus (\Gamma_1\cup\Gamma_2)\}.$$ Note that $M$ consists collisions on the complement of $\Gamma_1\cup\Gamma_2$.

 The cell $M_n$
(where $n$ is large) is made by trajectories that go deep into
the cusp and after exactly $m − 1$ bounces off its walls exit it.
We use the results and notation of \cite{BCD}. Let a cusp be made by two
boundary components $\Gamma_1, \Gamma_2$. Choose the coordinate system such that the equations of $\Gamma_1$ and $\Gamma_2$ are, respectively,
$y = g_1(x)$ and $y = −g_2(x)$, where $g_i$ are convex $C^3$ functions, such that $g_i(x) > 0$
for $x > 0$, and $g_i(0) = g′_i(0) = 0$ for $i = 1, 2$. The Taylor
expansion for the functions $g_i$ can be represented as $$g_i(x)=\frac{1}{2} a_i x^2+\cO(x^3)$$
where $a_i=fg'_i(0)>0$ is the curvature of $\Gamma_i$ at the point $p$. We denote $\bar a=(a_1+a_2)/2$.
Since the corner point $p$ is the end point of both $\Gamma_1$ and $\Gamma_2$, thus the coordinate $r$ takes two values at the cusp $p$, which we denote as $r', r''$, respectively. We denote $A_1=\{(r',\varphi), \varphi\in [-\pi/2, \pi/2]\}$ and $A_2=\{(r'',\varphi), \varphi\in [-\pi/2, \pi/2]\}$ as the two vertical boundaries in $\cM$ at the corner point $p$.

We let the induced billiard map
$F: M \rightarrow M$ and return time function $\R: M \rightarrow \NN$ be defined as previously mentioned.  Note that $M$ consists of all collisions on $\partial\cD\setminus(\Gamma_1\cup\Gamma_2)$, thus the measure of $M$ satisfies
\beq
\mu_{\cM}(M)=\frac{1}{2 |\partial\cD|}\int_0^{|\partial\cD|-|\Gamma_1|-|\Gamma_2|}\int_{-\pi/2}^{\pi/2} \cos\varphi\,d\varphi\,dr=\frac{|\partial\cD|-|\Gamma_1|-|\Gamma_2|}{|\partial\cD|}.
\eeq

To prove the CLT, we need to check conditions \textbf{(G1)-(G2)} and \textbf{(B1)-(B2)}.

 To estimate the constant $c_M$ in (\textbf{B2}), we need to estimate the measure of the set $(\R\geq n)$.  Note that the collision angle along the trajectory $T^k x$, $k=1,\cdots, n-1$, is invariant. Thus the set $\cup_{k=1}^{n-1} T^k M_n$  is squeezed between by a line with equation
 $$r^2= \frac{C}{n^2 \cos\varphi}+o(n^{-2})$$ for some constant $C>0$ and the line $\varphi=0$.  Thus
 \begin{align*}
 \mu_{\cM}(x\in \cM\,:\, \R\geq n)&=\frac{1}{ 4|\partial\cD|}\int_{-\pi/2}^{\pi/2}\left(\frac{1}{n \sqrt{\cos\varphi}}+o(n^{-2})\right)\,  \cos\varphi\,d\varphi\\
 &=\frac{1}{4|\partial\cD| n}\int_{-\pi/2}^{\pi/2}\sqrt{ \cos\varphi}\,d\varphi+o(n^{-1})\\
 &=\frac{\bar a}{2|\partial\cD| n}+o(n^{-1}),\end{align*}
 where we used the fact that
 $$\int_{-\pi/2}^{\pi/2}\sqrt{ \cos\varphi}\,d\varphi=2\bar a $$ as proved in \cite{BCD}.

Now we let $n\to\infty$ to get
 \beq c_M=\lim_{n\to\infty}n\mu_{\cM}(x\in \cM\,:\, \R\geq n) =\frac{\bar a}{2|\partial\cD| }. \eeq

Let $f$ be a piecewise  H\"{o}lder observable on $\cM$ with $\mu_{\cM}(f)=0$ and $\tf\in \cH_{\gamma}$ such that  $f$ is H\"{o}lder continuous on a small neighborhood of the singular set $A_1\cup A_2$.
Using \cite{BCD}, we can get  the estimation on the sum for $x\in M_n$,
$$\S_n f(x) =
\sum_{k=1}^n f(r_k,\varphi_k)=I_f\R(x)+E(x),$$ where $(r_k,\varphi_k)$ are
the standard coordinates of the reflection points, and $$I_f=\frac{\int_{-\pi/2}^{\pi/2} ( f(r',\varphi)+f(r'',\varphi))\sqrt{\cos\varphi}\, d\varphi}{4\bar a}$$
as well as $$E(x)\leq C\R(x)^{1-\gamma/2}$$ for some uniform constant $C>0$.

We define $E_{n,k}=(E\cdot \bI_{R<c_n})\circ T^k$, for $n\geq 1$, and $k=0,\cdots, n-1$.
Note that $$E_{n,k}\leq C c_n^{1-\gamma/2}=C (n\ln\ln n)^{1-\gamma/2}.$$

For any $n\geq 1$, $E_{n,0}$ is piecewise constant on $M$. It follows from results in \cite{CZ09} that $E_{n,0}$ enjoys exponential decay of correlations under the induced map $F$. More precisely, there exists $\theta_1\in (0,1)$ and constant $C>0$ such that for $i=0,\cdots, n-1$,
\beq\label{exdecaycE1}\cov(E_{n,0}\circ F^i \cdot E_{n,0})\leq C (n\ln\ln n)^{1-\gamma/2} \theta_1^i.\eeq
This implies that there exists $C_1>0$ such that
\[
\var(\sum_{k=0}^{n-1} E_{n,k})=n\var(E_{n,0})+\sum_{i=1}^n(n-i)\cov(|E_{n,0}\circ F^i \cdot E_{n,0}|)\leq C_1 (n\ln\ln n)^{1-\gamma/2}<cn
\] which verifies (\textbf{B3}).
Combining the above facts and Proposition \ref{propaCLT}, we have proved the following results.
\begin{theorem} Let $f\in \cH_{\gamma}$, with $\mu_{\cM}(f)=0$ and  $f$ is H\"{o}lder continuous on a small neighborhood of the singular set $A_1\cup A_2$. We assume $\int_{-\pi/2}^{\pi/2} ( f(r',\varphi)+f(r'',\varphi))\sqrt{\cos\varphi}\, d\varphi\neq 0$, then the  sequence
\beq\frac{f+\cdots +f\circ T^{n-1}-n\mu_{\cM}(f)}{ \sqrt{\sigma_f^2\cdot n \ln n}}\xrightarrow{d} \N(0,1)\eeq
converges in distribution, as $n\to\infty$, where
$$ \sigma_f^2=\frac{\left(\int_{-\pi/2}^{\pi/2} ( f(r',\varphi)+f(r'',\varphi))\sqrt{\cos\varphi}\, d\varphi\right)^2}{8\bar a|\partial\cD| }.$$
\end{theorem}

 Note that in \cite{BCD}, the authors also proved  the CLT for piecewise H\"{o}lder observable function $f$ with discontinuity contained in the singular set of $T^{\pm k}$, for some $k\geq 1$.  The improvement here is that  we only require that $f\in \cH_{\gamma}$ is H\"{o}lder continuous on any stable manifolds, and $f$ is H\"{o}lder continuous on a small neighborhood of the singular set $\{(r,0)\,:\, r\in A\}$. This improvement allows us to include more examples of processes $\{f\circ T^n\}$. One such example is that we can take a union of stable manifolds $\cW^s$, and a small neighborhood $U$ of $A_1\cup A_2$,  such that $f=\bI_{\cW^s} + \bI_U$. Then we still can apply the above CLT for this process $\{f\circ T^n\}$.

\subsubsection{Semi-dispersing billiards}

Billiards in a rectangle with a finite number of obstacles removed are known as semi-dispersing billiards.
We  consider a rectangle $\bR$ with dimensions $\bl_1$ by $\bl_2$, and place a finite number of obstacle $\bB=\cup_{i=1}^N \bB_i$  in the rectangle.
The table is obtained by removing these three obstacles $\cD=\bR\setminus \bB$.

The phase space $\cM$ is made of finite number of connected components. We define
the reduced phase space $M$ as made up only of collisions with the
convex obstacles. We let $\R: M \rightarrow \NN$ be defined as the first return time function that defines the induced map.  Note that $M$ consists all collisions on $\partial\cB$. Thus the measure of $M$ satisfies
\beq
\mu_{\cM}(M)=\frac{|\partial\bB|}{|\partial\cD|}.
\eeq

The induced map $F: M \rightarrow M$ is then equivalent to the well-studied Lorentz gas
billiard map on a torus with infinite horizon \cite{CZ05a}, which is known to have exponential decay of correlations; see \cite{CM}. More precisely,
the particle can move freely parallel to a unit
lattice vector without ever colliding with a scatterer;
this property is called {\it infinite horizon}.
 The structure of the $m$-cells $M_m = \{x\in M : \R(x) = m\}$ is
examined thoroughly in \cite{BSC90,BSC91,CM}; we will make use of some of the facts presented in
those references. It was proved by Sz\'asz and Varj\'u  in \cite{SV07} that
a non-classical central limit theorem is satisfied in this billiard.
It was proved in \cite{CZ08} that this system has a decay of correlations bounded by
$\text{const} \cdot n^{-1}$.

We will consider three configurations of the semidispersing billiards.\\

\noindent\textbf{Case I: Semidispersing billiard with one channel of free flights}\\

To simplify the geometry, we first put one convex obstacle $\bB_1$ at the center of the rectangle, and two quarter disks $\bB_2, \bB_3$ centered at the two adjacent end points of the rectangle. See Figure \ref{semifig}(a). For this type of billiards,  there is only one channel of free flight in the unfolding space of the billiard table by removing all flat sides of the boundary.

We define the $r$-coordinate such that $r=0$ at one end point of the rectangle. We denote  $A$ as the range of $r$-coordinates for all 2-periodic points $(r,0)$, whose trajectory are parallel to the horizontal channel, see  Figure \ref{semifig}(a). Note that $|A|/2$ is the width of the free flight channel.
  We now define
$$I_f=\frac{\int_{r\in A} f(r,0)\,d r}{|A|}.$$

For any $n\geq 1$,  for any $x\in M_{n}$, its trajectory is contained in the horizontal channel in the unfolding space. Let $J_f:=I_f \cR$. Clearly, $J_f$ satisfies condition \textbf{(B1)}.  Moreover, similar to the stadia case, we can get  the estimation on the sum for $x\in M_{n}$,
$$\tf(x) =
\sum_{k=1}^n f(r_k,\varphi_k)=J_f(x)+E(x),$$ where $(r_k,\varphi_k)$ are
the standard coordinates of the reflection points, and
$$E(x)\leq C\R(x)^{1-\gamma/2}$$ for some uniform constant $C>0$. This verifies (\textbf{B1}).

 To estimate the constant $c_M$, we need to estimate the measure of the set $(\cR>n)$.

  Note that for $x\in M_{n}$, its image will hit the boundary set with $r$-coordinates in $A$ for $n-1$ times before exiting the associated channel. Also note that the collision angle along the trajectory $T^k x$, $k=1,\ldots, n-1$, is invariant. Thus the set $\cup_{m=n}^{\infty}\cup_{k=1}^{m-1} T^k M_{m}$  is squeezed between by two parallel lines with equations:
 $$\sin\varphi= \pm \frac{|A|}{2 n \bl_1 }+o(n^{-2}).$$   By symmetric property of the billiard table, we have
 \begin{align*}
 \mu_{\cM}(x\in \cM\,:\,\cR>n)&=\frac{1}{ 4|\partial \cD|}\int_{r\in A}\left(\frac{|A|}{n\bl_1}+\cO(n^{-2})\right)\,  d r=\frac{|A|^2}{ 4n\bl_1|\partial \cD|}+\cO(n^{-2}).\end{align*}
 Note that
 $$ \mu_{\cM}(x\in \cM\,:\,\cR>n) =\sum_{k=0}^{n-1} \mu(F^k(x\in M\,:\,\cR>n))=n \mu_{\cM}(x\in M\,:\,\cR>n). $$
 This  implies that
 \begin{align*}
 \mu_{\cM}(x\in M\,:\,\cR>n)&=\frac{1}{n}\cdot \mu_{\cM}(x\in \cM\,:\,\cR>n)=\frac{|A|^2}{ 4 n^2\bl_1|\partial \cD|}+\cO(n^{-3}).\end{align*}
 Thus we have the following estimations:
\begin{align*}
H(n)&= \EE(J_f^2\bI_{ |J_f|<n})=I_f^2\sum_{m=1}^{n} m^2  \mu(M_{m})
=2I_f^2\sum_{m=1}^{n} m \mu(\cR>m).
\end{align*}

 Now we let $n\to\infty$ to get (\textbf{B2}),
\begin{align*} c_M\mu_{\cM}(M)=\lim_{n\to\infty}\frac{\mu_{\cM}(M)H(c_n)}{\ln n}=\frac{I_f^2|A|^2}{ \bl_1|\partial \cD|}.
\end{align*}

Many properties of the $m$-cells and of the induced billiard map in the semi-dispersing case
are quite similar to those in billiards with cusps. In particular, the measure of each $m$-cell is again
$\mu(M_m) \asymp m^{-3}$,
and as a result we once more have that the expectation of the return time
map $\R$ is finite. It is also known that for a point $x\in M_m$ we have $Fx \in M_n$, where
$$\O(\sqrt{m}) < n < \O(m^2).$$
 Moreover, for  the transition probabilities between cells, we have for admissible $n$ that
$$p_{n,m}:=\mu(Fx \in M_n | x \in M_m) \asymp \frac{m+n}{n^3}.$$
From this it is clear that semi-dispersing billiards are square spreading.
Note that
\begin{eqnarray*} \sum\limits_{n=\sqrt{m}}^{m^2} n \mu(F x \in M_n | x \in M_m) =\cO(\sqrt{m}).
\end{eqnarray*}

 We have, for any  $m$ large,
\begin{align}\label{eq:linspr6}
\EE((I_f\cR\cdot \bI_{I_f\cR<c_n})\circ F(x) | \cR(x) = m) &= \sum\limits_{n\in\B_m} I_f \sum\limits_{n=\sqrt{m}}^{m^2} n \mu(F x \in M_n | x \in M_m)\cdot \1_{M_m}(x)=e_m\1_{M_m}(x),
\end{align}
where  $e_m=\cO(\sqrt{m})$. This implies that we can take $\theta=0$ and  $\cE_{n,0}=\sum_{m=1}^{c_n^2} e_m \bI_{M_m}$. This verifies (\textbf{B3}).

Next, we will prove (\textbf{B4}).

Let $\cM_{i,j}=\cup_{m=i}^j M_m$  be the union of cells with indices satisfying $1\leq i\leq j$.
Denote
\beq\nn \label{eq: function sequence2}
{\tf}_{i,j}:=\tf\cdot \bI_{M_{i,j}}-\mu(\tf\cdot \bI_{\cM_{i,j}})
\eeq
\begin{proposition}[Exponential decay of correlations for $\tf_{i,j}$]\label{boundedcovtf3}
Let $\tf$ be an induced function on $M$, and $\tf_{i,j}$ is as defined above for any $1\leq i\leq j\leq \infty$. Then
for all $k$ sufficiently large,
\beq\label{eq:condition D_r(u_n)2f3}
|\mu({\tf}_{i,j}\circ F^k\cdot{\tf}_{i,j})-\mu({ \tf}_{i,j})^2|\leq C\theta^k.
\eeq
where  $\theta\in (0,1)$ is a constant.
\end{proposition}


\begin{proof}
For any  set $M_m$ in $\cM_{i,j}$, we foliate it into unstable curves  that stretch completely from one side to the other. Let $\{W_{\beta}, \beta\in \cA, \lambda\}$ be the foliation, and $\lambda$ the factor measure defined on the index set $\cA$.  This enables us to define a standard family, denoted as $\cG_{i,j}=(\cM_{i,j},
\mu_{i,j})$, where $\mu_{i,j}:=\mu|_{\cM_{i,j}}$. 

Our first step in proving the decay of correlations is to estimate the $\cZ$ function of $F^k\cG_{i,j}$.
 For semi-dispersing billiards, $FM_m$ is a strip that has length $\sim m^{-1/2}$ and width $\sim m^{-2}$. Also, by construction, the singular set is close to grazing collisions, so the density of $\mu_{i,j}$ is of order $\sim m^{1/2}$ on $ FM_m$. Thus we obtain for $j\leq \infty$,
\begin{align*}
\cZ(F\cG_{i,j})&=\mu(\cM_{i,j})^{-1}\int_{\beta\in\cA}|FW_{\beta}|^{-1}\,\lambda_{\cG_{i,j}}(d\beta)\\
&\le C
\mu(\cM_{i,j})^{-1}\cdot \sum_{m=i}^{j} m^{1/2}\cdot m^{-3} \\
&= C_1 i^2\cdot i^{-3/2}=C_1 i^{1/2}.\end{align*}

Note that for any large $m<l$, we have
$$\mu(F\cM_{i,j}\cap F\cM_{m,l})\leq F_*\mu_{i,j}(r<\eps_{m,l})$$
where $\eps_{m,l}$ is approximately the width of the smallest cell in $\cM_{m,l}$, which is of order $m^{-2}$.
Using  Lemma \ref{growthlemma2}, we have that
\begin{align}\label{estMiMk1}
\mu(F^k(\cM_{i,j})\cap \cM_{m,l})&=\mu(F^{n+1}(\cM_{i,j})\cap F\cM_{m,l})\nonumber\\
&\leq F^k_*\mu_{i,j}(r<\eps_{m,l})\nonumber\\
&\leq C'(\vartheta^{k-1}\cZ(FG_{i,j})+C'')\eps_{m,l}\mu(\cM_{i,j})\nonumber\\
&\leq C(C_1\|f\|_{\infty}\vartheta^{k-1}i^{1/2}+C'')m^{-2} i^{-2}
\end{align}

For any fixed large $k$, we truncate $ \tf_{i,j}$ at two extra levels, with $i\leq p<q\leq j$, which will be chosen later, i.e.
$$ \tf_{i,j}= \tf_{i,p}+ \tf_{p,q}+ \tf_{q,j}$$

The function $ \tf_{i,q}= \tf_{i,p}+ \tf_{p,q}$ is bounded with $\|\tf_{i,q}\|_{\infty}\leq C\|f\|_{\infty} q$, and H\"{o}lder norm $\| \tf_{i,q}\|_{\gamma}\leq C\|f\|_{\gamma} q^{1+\gamma}$. Thus by (\ref{correT}), we know that
\beq\label{iqiq2}
\mu( \tf_{i,q}\circ F^k,  \tf_{i,q})\leq C\| \tf_{i,q}\|_{C^{\gamma}}^2 \theta^k +\mu( \tf_{i,q})^2\leq C q^{2+2\gamma}\theta^k +\cO(q^{-2})
\eeq
where we used the fact that
$$\mu( \tf_{i,q})=-\mu( \tf_{q,j})=\cO(q^{-1})$$

Next we estimate
\begin{align*}
&\mu( \tf_{i,p}, \tf_{q,j}\circ F^k)\leq C\|f\|_{\infty}^2\sum_{m=i}^p\sum_{l=q}^j m\cdot l\cdot \mu(M_l\cap F^k M_m)\\
&=C\|f\|_{\infty}^2\sum_{m=i}^p\sum_{l=q}^j \sum_{t=1}^m \sum_{s=1}^l \mu(M_l\cap F^k M_m)\\
&=C\|f\|_{\infty}^2\sum_{m=i}^p\sum_{t=1}^m
\left(\sum_{s=1}^q \sum_{l=q}^j  \mu(M_l\cap F^k M_m)+\sum_{s=q}^j \sum_{l=s}^j  \mu(M_l\cap F^k M_m)\right)\\
&=C\|f\|_{\infty}^2\sum_{m=i}^p\sum_{t=1}^m
\left(q\cdot \mu(\cM_{q,j}\cap F^k M_m)+\sum_{s=q}^j   \mu(\cM_{s,j}\cap F^k M_m)\right)\\
&=C\|f\|_{\infty}^2\sum_{t=1}^i
\left(q\cdot \mu(\cM_{q,j}\cap F^k \cM_{i,p})+\sum_{s=q}^j   \mu(\cM_{s,j}\cap F^k \cM_{i,p})\right)\\
&+C\|f\|_{\infty}^2\sum_{t=i}^p
\left(q\cdot \mu(\cM_{q,j}\cap F^k \cM_{t,p})+\sum_{s=q}^j   \mu(\cM_{s,j}\cap F^k \cM_{t,p})\right)\\
&=C\|f\|_{\infty}^2
\left( (C_1\|f\|_{\infty}\vartheta^{k-1}i^{1/2}+C'')q^{-1} i^{-2}+\sum_{s=q}^j (C_1\|f\|_{\infty}\vartheta^{k-1}i^{1/2}+C'')s^{-2} i^{-2} \right)\\
&+C\|f\|_{\infty}^2\sum_{t=i}^p
\left((C_1\|f\|_{\infty}\vartheta^{k-1}t^{1/2}+C'')q^{-1} t^{-2}+\sum_{s=q}^j  (C_1\|f\|_{\infty}\vartheta^{k-1}t^{1/2}+C'')s^{-2} t^{-2}\right)\\
\end{align*}
Then one can check that
\beq\label{lesalpha02}
\mu( \tf_{i,p}, \tf_{q,j}\circ F^k)=\cO(q^{-1}(C_1\vartheta^k i^{-1/2}+C_2))
\eeq

Similarly, we can show that
\beq\label{lesalpha00}
\mu( \tf_{q,j}, \tf_{i,p}\circ F^k)=\cO(i^{-1}(C_1\vartheta^k q^{-1/2}+C_2))
\eeq

Next, we estimate
\begin{align*}
&\mu( \tf_{p,j}, \tf_{p,j}\circ F^k)\leq C\|f\|_{\infty}^2\sum_{m=p}^j\sum_{l=p}^j m\cdot l\cdot \mu(M_l\cap F^k M_m)\\
&=C\|f\|_{\infty}^2\sum_{t=1}^p
\left(p\cdot \mu(\cM_{p,j}\cap F^k \cM_{p,j})+\sum_{s=p}^j   \mu(\cM_{s,j}\cap F^k \cM_{p,j})\right)\\
&+C\|f\|_{\infty}^2\sum_{t=p}^j
\left(p\cdot \mu(\cM_{p,j}\cap F^k \cM_{t,j})+\sum_{s=p}^j   \mu(\cM_{s,j}\cap F^k \cM_{t,j})\right)\\
&=C\|f\|_{\infty}^2
\left( (C_1\|f\|_{\infty}\vartheta^{k-1}p+C'')p^{-2} +\sum_{s=p}^j (C_1\|f\|_{\infty}\vartheta^{k-1}p+C'')s^{-2} p^{-1} \right)\\
&+C\|f\|_{\infty}^2\sum_{t=p}^j
\left((C_1\|f\|_{\infty}\vartheta^{k-1}t+C'')p^{-1} t^{-2}+\sum_{s=p}^j  (C_1\|f\|_{\infty}\vartheta^{k-1}t+C'')s^{-2} t^{-2}\right)\\
\end{align*}

Then one can check that
\beq\label{lesalpha0}
\mu( \tf_{p,j}, \tf_{p,j}\circ F^k)=\cO(p^{-2}(\vartheta^k p+C_1))+\cO(p^{-1}\vartheta^k \ln p)
\eeq

Combining the above estimations, we have
\begin{align*}
\mu( \tf_{i,j}, \tf_{i,j}\circ F^k)&\leq C_1 \vartheta^k(q^{2+2\gamma}+p^{-1} \ln p)+C_2 p^{-1}\end{align*}

Now we choose $q=\vartheta^{-k/d}$, $p=\sqrt{q}$, where $d=6(1+\gamma)$. Then above estimations implies that
$$\mu( \tf_{i,j}, \tf_{i,j}\circ F^k)\leq C_1 \theta^k$$
where $\theta=\vartheta^{2/3}$.
\end{proof}

 Combining the above facts and Proposition \ref{propaCLT}, we have proved the following result.

\vspace{0.5cm}

\noindent\textbf{Case II. Semidispersing billiards with $N$ channels of free flights} \\ \\

 Assume the semi-dispersing table has exactly $N$  channels of free flights in the unfolding space of the billiard table by removing all flat sides of the boundary.  For simplicity, we assume the table is a square with  length $\bl$.

\begin{figure}[h]
\psfrag{A1}{$A_1$} \psfrag{A}{$A$} \psfrag{A2}{$A_2$}
\psfrag{x0}{$x_0$} \psfrag{n}{$\varphi$}
\psfrag{y0}{$y_0$}
\psfrag{I1}{$I_1$} \psfrag{(a)}{$(a)$}\psfrag{(b)}{$(b)$} \centering
 \includegraphics[height=3in]{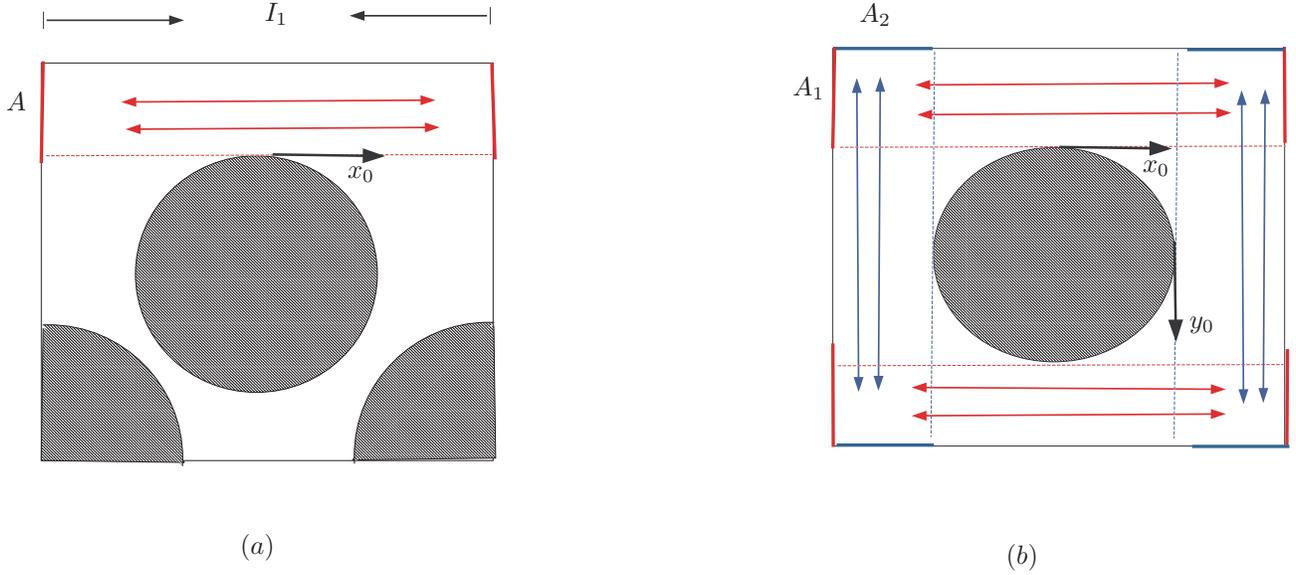}
 \caption{Two types of semidispersing billiards}\label{semifig}
 \end{figure}

For any $i\in \{1,\ldots, N\}$, we consider the $i$-th channel in the unfolding space, which is defined as the unbounded region  bounded by two parallel  tangent lines $l_i^1, l_i^2$ of the scatterers, such that  any lines parallel to $l_i^1$ in the channel will not touch any scatterers. Let $\bx_i$ be a tangent vector of the scatterers, such that its  trajectory coincides with  $l_i^1$. Note that $T \bx_i$ is based on a straight side of the square.  Let $\varphi_i$ be the collision angle of $T\bx_i$, which is formed by the line $l_i^1$ and the  straight side of the square. Let $\bI_i$ be the length of the segment of $I_i^1$ in the rectangle.  We also call $\varphi_i$ \textit{the angle of the $i$-th channel}.

 We denote  $A_i$, $i=1,\ldots,N$ as the range of $r$-coordinates on the boundary of the square, such that for any  point $(r, \varphi_i)$,  with $r\in A_i$, its trajectory is parallel to the $i$-th channel.  Note that $|A_i|/2$ is the width of the corresponding channel.
   We now define
$$a_i=\frac{\int_{r\in A_i} f(r,\varphi_i)\,d r}{|A_i|},\,\,\,\,\,\,\,i=1,\ldots,N.$$

For any $n\geq 1$, we decompose $M_n=\cup_{k=1}^N M_{n,k}$ into $N$ components, such that for any $x\in M_{n,k}$, its trajectory is contained in the $k$-th channel in the unfolding space.  Let $$U_{n,i}:=\cup_{m\geq n} M_{m,i}$$ be the collection of all collision vectors in $\cM$, whose forward trajectories experience at least $n$-iterations in the $i$-th channel.

We now define  $$J_f:=\sum_{i=1}^N a_i \cdot \R\cdot \bI_{U_{1,i}}.$$ Clearly, $J_f$ satisfies condition \textbf{(B1)}.  Moreover, similar to  case I, we can get  the estimation on the sum for $x\in M_{n,i}$,
$$\tf(x) =
\sum_{k=1}^{n-1} f(T^kx)=J_f(x)+E_i(x)$$ where
$$E_i(x)\leq C\R(x)^{1-\gamma/2}$$ for some uniform constant $C>0$.

 To estimate the constant $c_M$, we need to estimate the measure of the set $U_{n,i}$.
   Note that for $x\in M_{n,i}$, its image will hit the boundary set with $r$-coordinates in $A_i$ for $n-1$ times before exiting the associated channel. Also note that the collision angle along the trajectory $T^k x$, $k=1,\cdots, n-1$, is invariant. Thus the set $\cup_{m\geq n}\cup_{k=1}^{m-1} T^k M_{m,i}$  is squeezed between by two parallel lines with equations:
 $$\sin\varphi= \sin\varphi_i\pm \frac{|A_i|}{2 n \bl_i }+o(n^{-2}).$$   By the symmetric property of the billiard table, we have
 \begin{align*}
 \mu(U_{n,i})&=\frac{1}{ 4n|\partial \cD|}\int_{r\in A_i}\left(\frac{|A_i|}{n\bl_i}+\cO(n^{-2})\right)\,  d r=\frac{|A_i|^2}{ 4n^2\bl_i|\partial \cD|}+\cO(n^{-2}),\end{align*}
 where we have used the fact that
 $$ \mu(\cup_{m\geq n}\cup_{k=1}^{m-1} T^k M_{m,i}) =\sum_{k=0}^{n-1} \mu_{\cM}(F^kU_{n,i})=n \mu(U_{n,i}). $$

 Note that
\begin{align*} c_M&=\lim_{t\to\infty}\frac{ \EE(J_f^2\bI_{ |J_f|<t})}{2\ln t}=\lim_{t\to\infty}t^2\,\mu(|J_f|\geq t)\nonumber\\
&=\lim_{t\to\infty}\sum_{k=1}^{N}t^2  \cdot \mu(U_{t/a_k,k}) =\sum_{k=1}^{N}\frac{a_k^2|A_k|^2}{4\bI_i |\partial D|\mu_{\cM}(M)}.\end{align*}

 Combining the above facts and Proposition \ref{propaCLT}, we have proved the following results.
\begin{theorem}
Assume the unfolding space of the semidispersing billiards has $N$ channels of free flights. Let $f\in \cH_{\gamma}$ be H\"{o}lder continuous on a small neighborhood of the singular set $\{(r,\varphi_i)\,:\, r\in A_i, i=1, \cdots, N\}$.   Assume $\sum_{k=1}^N(\int_{r\in A_i} f(r,\varphi_i)\,dr)^2\neq 0,$ then the  sequence
$$\frac{f+\cdots +f\circ T^{n-1}-n\mu(f)}{ \sqrt{  \sigma_f^2\cdot n \ln n}}\xrightarrow{d} N(0,1)$$
converges in distribution, as $n\to\infty$, with $$\sigma_f^2=\sum_{i=1}^N\frac{(\int_{r\in A_i} f(r,\varphi_i)\,dr)^2}{ 4 \bI_i |\partial \cD|}.$$
\end{theorem}

\section*{Acknowledgements}

 H.Z. is supported in
part by NSF CAREER grant DMS-1151762, as well as  a grant from the Simons
Foundation (337646, HZ).

\end{document}